\documentclass{amsart}
\usepackage{amssymb}

\usepackage{graphicx}
\usepackage{amscd}

\newtheorem{theorem}{Theorem}
\theoremstyle{plain}

\newtheorem{corollary}{Corollary}

\newtheorem{example}{Example}

\newtheorem{lemma}{Lemma}

\newtheorem{remark}{Remark}

\numberwithin{equation}{section}

\begin{document}
\title[Asymptotic Normality of Extreme values Statistics]{Multivariate Normality of a class of statistics based on extreme observations}
\author[G. S. LO]{Gane Samb LO}
\address{LSTA, UPMC, France and LERSTAD, Universit\'e Gaston Berger
de Saint-Louis, SENEGAL}
\email{gane-samb.lo@ugb.edu.sn, ganesamblo@ufrsat.org}
\urladdr{www.lsta.upmc.fr/gslo}

\keywords{Extreme value theory, order statistics, empirical distribution function,
domain of attraction of the maximum, characterization.\\ \\
This paper is part of the Doctorate of Sciences of the author, Dakar University 1991, under the title :
\textit{Empirical characterization of the extremes : The asymptotic normality of the
characterizing vectors}.}

\begin{abstract}
\Large
Let $X_{1},X_{2},...$ be a sequence of independent random variables ($rv$)with
common distribution function ($df$) $F$ such that $F(1)=0$ and for each $n\geq 1,$
let $X_{1,n}\leq X_{2,n}\leq ...\leq X_{n,n}$ denote the order statistics
based on the n first of these random variables. L\^{o} (\cite{gslod})
introduced a class of statistics aimed at characterizing the asymptotic behavior of the univatiate extremes. This class this estimator of the square of the extremal index of a $df$ lying in the extremal domain of attraction :

$$
k^{-1}\sum_{j=\ell +1}^{j=k}\ \sum_{i=j}^{i=k}i(1-\delta _{ij}/2)\left( \log
X_{n-i+1,n}-\log X_{n-i,n}\right) 
$$

$$
\times \left( \log X_{n-j+1,n}-\log
X_{n-j,n}\right) ,
$$

\noindent where $(k,\ell )$ is a couple of integers such that $k\rightarrow +\infty $, 
$k/n\rightarrow 0$, $\ell ^{2}/k\rightarrow 0$, as $n\rightarrow
0\rightarrow ,$ $log$ stands for the natural logarithm and $\delta _{ij\text{
}}$is the Kronecker symbol.\\

\noindent In total $\mathbb{R}^{8}$-vectors are used in this paper and include the most popular
statistics used in the literature. We consider here a multivariate approach and provide the 
asymptotic laws of such vectors. This allows quickly finding asymptotic laws of functional of new statistics and new estimators of the extremal index such as the Dekkers \textit{et al.} (\cite{dehf}) and Hasofer and Wang statistics (\cite{hwf}) for example as in \textit{Bah et al.} (\cite{balothiam}) 
\end{abstract}

\maketitle
\Large
\section{Introduction} \label{sec1}

\noindent L\^{o} (\cite{gslod}) characterized the class of distribution functions
($df$) $F$ attracted to some nondegenerated $df$ $M$ (written $F\in D(M)$) by
four statistics while no condition was required on $F$. This empirical and unified approach includes detection procedures of the extremal law of a sample and statistical tests. In both cases, one has to determine the
limiting laws of such characterizing statistics for the Extremes.\\

\noindent At the same time, a considerable number of statistics, including the celebrated Hill statistics \cite{hillf} in 1975, have been introduced and studied by many authors. we may cite among them the works of Hall (1981) \cite{hall1}, Beirlant and Teugels (1986) \cite{berf}, Deheuvels and Mason (1985) \cite{dmf3}, Deheuvels and Mason (1990) \cite{dmf1}, Deheuvels, Haeusler and Mason (1990) \cite{deheuvels-haeusler-mason}, Lo \cite{gsloa}, \cite{gsloc}, \cite{gslod}, etc. who gave all the  limits in probability, as almost sure limits and asymptotic laws of Hill statistics. Dekkers et al. (1989) introduced the moment estimor \cite{dehf} while Hasofer and Wang \cite{hwf} gave another interesting estimator.\\

\noindent It seems that a joint multivariate of the largest class of the used statistics is yet to be done. Such results allow to easily rediscover Dekkers \textit{et al.} \cite{dehf} results and to find asymptotic of further statistics as done for Hasofer and Wang statistics in \cite{balothiam}. They also allow many combinations to form new estimators while quickly providing the asymptotic normality and next statistical tests.\\

\noindent The idea is to use the approximation the sequence empirical processes related with the data to a sequence of Brownian Bridges as in the so-called Hungarian construction in \cite{cchm}. We then express all the results with the help of the same sequence of Brownian Bridges and get the multivariate asymptotic normality.\\

\noindent The reader is referred to L\^{o} \cite{gsloa} for a general introduction to this paper and to Leadbetter and Rootz\`{e}n \cite{leadbetter} and Resnick (\cite{resnick}), Galambos \cite{galambos}, de Haan \cite{dehaan} and Beirlant \textit{et al.} \cite{bgt} for detailed references on extreme value theory. However, we recall that is attracted to some non generated $df$ (denoted $F\in D(M)$), then M is necessarily th Gumbel type of $df$ : 
\begin{equation*}
\Lambda (x)=\exp (-\exp (-x)),\text{ }x\in \mathbb{R},
\end{equation*}
or the Fr\'{e}chet type of $df$ of parameter $\gamma >0,$
\begin{equation*}
\phi _{\gamma }(x)=\exp (-^{-\gamma })\chi _{\left[ 0,+\infty \right[ }(x),%
\text{ }x\in \mathbb{R}\ 
\end{equation*}
or the Weibull type of d.f of parameter $\gamma >0$

\begin{equation*}
\psi _{\gamma }(x)=\exp (-(x)^{-\gamma })\chi _{\left] -\infty ,0\right]
}(x)+(1-1_{\left] -\infty ,0\right] }(x)),\ \ x\in \mathbb{R},\ 
\end{equation*}
where $\chi _{A}$ denotes the indicator function of the set A.

\bigskip

\noindent Several analytic characterizations of $D(\phi )=\cup _{\gamma >0}D(\phi
_{\gamma }),$ $D(\psi )=\cup _{\gamma >0}D(\psi _{\gamma }),$ $D(\Lambda )$
and $\Gamma =D(\Gamma )\cup D(\phi )\cup D(\psi )$ exist. We quote here only
those of them involved in our present work.

\begin{theorem}[A] \label{theo1}
We have :

\begin{enumerate}
\item  Karamata's representation (KARARE)

\noindent (a) $F\in D(\phi _{\gamma }),$ $\gamma >0$, iff 
\begin{equation}
F^{-1}(1-u)=c(1+f(u))u^{-1/\gamma }\exp (\int_{u}^{1}b(t)t^{-1}dt),\text{ }%
0<u<1,  \label{1.1}
\end{equation}
where $\sup (\left| f(u)\right| ,\left| b(u)\right| )\rightarrow 0$ as $%
u\rightarrow 0$ and c is a positive constan and $F^{-1}(1-u)=\inf
\{x,F(x)\geq u\},$ $0\leq u\leq 1,$ is the generalized \ inverse of F with $%
F^{-1}(0)=F^{-1}(0+)$.\newline
\newline
\noindent (b) $F\in D(\psi _{\gamma }),$ $\gamma >0$, iff $x_{0}(F)=\sup
r\{x,$ $F(x)<1\}<+\infty $ and 
\begin{equation}
F^{-1}(1-u)=c(1+f(u))u^{1/\gamma }\exp (\int_{u}^{1}b(t)t^{-1}dt),\text{ }%
0<u<1, \label{1.2}
\end{equation}
where $c$, $f(\cdot )$ and $b(\cdot )$ are as in (\ref{1.1})

\item  Representation of de Haan (Theorem 2.4.1 in \cite{dehaan})

\noindent $F\in D(\Lambda )$ iff 
\begin{equation}
F^{-1}(1-u)=d-s(u)+\int_{u}^{1}s(t)t^{-1}dt,\text{ }0<u<1,  \label{1.3a}
\end{equation}

\noindent where d is a constant and $s(\cdot )$ is admits this KARARE : 

\begin{equation}
s(u)=c(1+f(u))\exp (\int_{u}^{1}b(t)t^{-1}dt),\text{ }0<u<1,
\label{1.3b}
\end{equation}
$c$, $f(\cdot )$ anf $b(\cdot )$ being defined as in (\ref{1.1}).
\end{enumerate}
\end{theorem}

\bigskip
\noindent These representations are closely related to slowly and regularly varying functions at zero. Recall that a functions $s(u)$, $0<u<1$, is \textbf{S}lowly \textbf{V}arying function at \textbf{Z}ero (SVZ) if for all $\lambda>0$,
$$
\lim_{u\rightarrow 0}s(\lambda u)/s(u)=1.
$$
It is a \textbf{R}egularly \textbf{V}arying function at \textbf{Z}ero (SVZ) with exponent $d > 0$ if for for all $\lambda>0$,
$$
\lim_{u\rightarrow 0}s(\lambda u)/s(u)=\lambda^{d}.
$$

\bigskip

\noindent In fact  (\ref{1.1}) is the representation of a $RVZ$ function of exponent $-1/\gamma$ and (\ref{1.3b}) is that of a $SRZ$ function.\\

\noindent Now let $X_{1},X_{2},...$ be a sequence of independent copies (s.i.c) of a
real random variable ($rv$) $X$ with $df$ $F(x)=\mathbb{P}(X\leq x)$. Being
only concerned by the upper tail of F, we assume WLOG that $X\geq 1$ and
define a s.i.c. of the $rv$ $Y=\log X$ denoted $Y_{1},Y_{2},...$ with $df$ $%
G(x)=P(Y\leq x)=F(e^{x})$, $x\geq 0.$ Finally $Y_{1,n}=\log
X_{1,n}\leq ...\leq Y_{n,n}=\log X_{n,n}$ are their respective order
statistics.\\

\noindent We are investigating the joint asymptotic laws of the following statistics.\\

\begin{equation*}
A_{n}(1,k,\ell )=k^{-1}\sum_{j=\ell +1}^{j=k}\ \sum_{i=j}^{i=k}i\delta
_{ij}\left( Y_{n-i+1,n}-Y_{n-j+1,n}\right) \left(
Y_{n-j+1,n}-Y_{n-j,n}\right)
\end{equation*}
\begin{equation*}
T_{n}(2,k,\ell )=\frac{1}{k}\sum_{j=\ell +1}^{j=k}j\left(
Y_{n-j+1,n}-Y_{n-j,n}\right) ;Y_{n-k,n},0\leq k<n;
\end{equation*}

\begin{equation*}
T_{n}(1,k,\ell )=T_{n}(2,k,\ell )A_{n}(1,k,\ell )^{-1/2},
\end{equation*}
\ \ \ \ \ \ \ \ \ 
\begin{equation*}
T_{n}(4)=Y_{n,n},\text{ T}_{n}\text{(5)}=T_{n}(2,\ell ,1),\text{ }%
T_{n}(6)=T_{n}(2,\ell ,1)/\left( Y_{n-\ell ,n}-Y_{n-k,n}\right) ,
\end{equation*}
\begin{equation*}
T_{n}(7)=A_{n}(1,\ell ,1)/(Y_{n-\ell ,n}-Y_{n-k,n})^{2},\text{ }%
T_{n}(8,v)=n^{-\nu}\left( Y_{n-\ell ,n}-Y_{n-k,n}\right) ^{-1}
\end{equation*}
and finally for $y_{0}=x_{0}(G)$,
\begin{equation*}
T_{n}(9)=(y_{0}-Y_{n-\ell ,n})/\left( y_{0}-Y_{n-k,n}\right)
\end{equation*}
where $k$ and $\ell$ are integers such that $1\leq \ell < k < n$,  $\nu$ is any real positive number and $T_{n}(9)$ is only defined when $x_{0}(G)<+\infty .$

\bigskip

\bigskip 
\noindent As a reminder, we recall the second theorem corresponding to the necessity part of
the characterization of L\^{o} (\cite{gslod}), that is the convergence of some vectors formed from the previous ones simply imply that $F$ lies in some extremal domain, depending on the value of the limits.

\begin{theorem}[B] \label{theo2}
Let $k=[n^{\alpha }],$ $\ell =[n^{\beta }],$ $0.5<\beta <\alpha <1,$ $2\nu
=\min (1-\alpha ,\alpha +\delta -1)>0$ for $0<\delta <0.5.$

\begin{itemize}
\item[(a)]  If $\mathbb{T}_{n}(1)=(T_{n}(1),...,T_{n}(7),$ $T_{n}(8,\beta
/2))\rightarrow _{P}(1,d,0,y_{0},d,0,0,1)$, then

\begin{itemize}
\item[(i)]  $F\in D(\Lambda )$ whenever d=0 and

\item[(ii)]  $F\in D(\phi _{1/d})$ whenever $0<d<+\infty $ and necessarily $%
y_{0}=+\infty .$
\end{itemize}

\item[(b)]  If T$_{n}(4)\uparrow y_{0}<+\infty $ and $\mathbb{T}%
_{n}(2)=(T_{n}(1), $ $...$, $T_{n}(7),$ $T_{n}(9))\rightarrow
_{P}(c,0,0,y_{0},0,0,0)$ with $1<c<\sqrt{2},$ then $F\in D(\psi _{\gamma })$
for $\gamma =-2+c/(c^{2}-1).$
\end{itemize}
\end{theorem}

\bigskip

\noindent This motivates a systematic investigation of the limit laws of the ECSFEXT.
We do not include $T_{n}(4)$ in our study for its case is classical in
extreme value theory. Namely, we must find non random sequences $\sigma
_{n}(i,k,\ell )$ and $\mu _{n}(i,k,\ell )$ such that 
\begin{equation*}
\sigma _{n}(i,k,\ell )(STAT_{i}-\mu_{n}(i,k,\ell )\rightarrow _{d}NDD,
\end{equation*}
where $STAT_{i}$ is an element of the $ECSFEXT$, a ratio or a vector of its
elements, $\rightarrow _{d}$ denotes the convergence in distribution and NDD
is some nondegenerate distribution. It will be seen further that \ NDD is,
in most cases, a Gaussian $rv$ or an extremal one in some others.

\bigskip

\noindent Let us now classify the elements of $\Gamma $ in our convinience. From (\ref
{1.1}), (\ref{1.2}), (\ref{1.3a}), it is clear that for each of the three
domains, the couple $(f,b)$ represents a subset having elements only
distinguishable by constants. We then write this class $(f,b)$ (cf. Lemma 1
in L\^o (\cite{gslod})) :

$F\equiv (f,b)\in D(\Lambda )$ iff 
\begin{equation}
G^{-1}(1-u)=d-s(u)+\int_{0}^{1}s(t)t^{-1}dt,\text{ }0<u<1, \label{1.4}
\end{equation}
where $s(\cdot )$ satifies (\ref{1.3b});\\

\noindent $F\equiv (f,b)\in D(\phi _{\gamma })$ iff 
\begin{equation}
G^{-1}(1-u)=-(\log u)/\gamma +\log c+\log (1+f(u))+\int_{0}^{1}b(t)t^{-1}dt,%
\text{ }0<u<1; \label{2.5}
\end{equation}
\bigskip

\noindent Finally $F\equiv (f,b)\in D(\psi_{\gamma})$ iff $y_{0}=x_{0}(G)=\sup \{x$, $%
G(x)<1\}<+\infty $ and 
\begin{equation}
y_{0}-G^{-1}(1-u)=c(1+f(u))\text{ }u^{1/\gamma }\exp
(\int_{u}^{1}b(t)t^{-1}dt),\text{ }0<u<1. \label{2.6}
\end{equation}
\bigskip

\noindent Our results below will show that one has asymptotic normality with no
condition on $f$ nor on $b$ but with \textit{random centering coefficients} that is 
\begin{equation*}
\sigma _{n}(i,k,\ell )(STAT_{i}-\mu _{n}(i,\widetilde{k},\ell )\rightarrow
_{d}NDD,
\end{equation*}
where $\widetilde{k}$ is random and satisfies $\widetilde{k}/k\rightarrow
_{P}1$ as $n\rightarrow \infty .$ But when we attemt to have non random
centering sequences $\mu _{n}(i,k,\ell )$, only $f$ makes problems. In early
studies of $T_{n}(2,k,\ell )$ for instance, Hall(1982) (\cite{hall1}), Cs\H{o}rg\"o and
Mason(1985) (\cite{csorgo-mason}), Lo(1989) (\cite{gsloc}) imposed f=0. In particular Cs\H{o}rg\"o and
Mason(1985) (\cite {csorgo-mason}), L\^{o}(1989) (\cite{gsloc}) respectively denoted $D^{\ast }(\phi
) $ and $D^{\ast }(\Lambda )$ the obtained classes.\\

\noindent This condition is not too restrictive at all for that, for exemple, $D^{\ast
}(\Lambda )$ includes any element of $D(\Lambda )$ having an ultimate
derivative, that is the most important cases : normal, lognormal,
exponential, etc. Nevertheless, we characterize here the possible limits and
then, obtain simple results under general conditions like 
\begin{equation}
f^{\prime }(u)=\frac{\partial f}{\partial u}(u)\text{ exists and }uf^{\prime
}(u)\rightarrow 0\text{ as }u\rightarrow 0,  \label{1.7}
\end{equation}
or 
\begin{equation}
\sqrt{k}\max (f(k),f(U_{k,n}))\rightarrow _{P}0,  \label{1.8}
\end{equation}
or 
\begin{equation}
f(u)=f_{1}(u)(1+f_{2}(u))  \label{1.9}
\end{equation}
were $f_{1}$ satisfies (\ref{1.7}) and 
\begin{equation*}
\sqrt{k}f_{1}(k/n)\max (f_{2}(k/n),f(U_{k,n}))\rightarrow _{P}0,
\end{equation*}

\noindent where in all these conditions $U_{k,n}$ is the $kth$ maximum among n
independent $rv's$ uniformly distributed on (0,1) (see (2.2) below). Let $%
\Gamma (0),$ $\Gamma (1)$ and $\Gamma (2)$ respectively the subclasses of $%
\Gamma $ satisfying (\ref{1.7}), (\ref{1.8}) or (\ref{1.9}). Each of them is
quoted as a regularity condition.\\

\noindent Each single statistic is systematically treated apart in Sections \ref{sec3}, \ref{sec4} and
\ref{sec5} while the ratios are studied in Section \ref{sec6}. Section \ref{sec7} is devoted to
multivariate limit laws as our best achievements. All the results expressed in the same probability space through the same of Brownian brodges $B_{1}$, $B_{2}$, etc. We therefore begin to define this probability space.

\section{Description of the limiting laws} \label{sec2}

\noindent Cs\"{o}rg\H{o} \textit{et al.} (see \cite{cchm}) have constructed a probability
space holding a sequence of independent uniform random variables $U_{1},\
U_{2},$\ ... and a sequence of Brownian bridges $B_{1},B_{2},...$\ such that
for each $0<\nu <1/4$, $as$\ $n\rightarrow \infty ,$%
\begin{equation}
\underset{1/n\leq s\leq 1-1/n}{\sup }\frac{\left| \sqrt{n}%
(U_{n}(s)-s)-B_{n}(s)\right| }{(s(1-s))^{1/2-\nu }}=O_{p}(n^{-\nu })
\label{2.1}
\end{equation}
and 
\begin{equation}
\underset{1/n\leq s\leq 1-1/n}{\sup }\frac{\left| B_{n}(s)-\sqrt{n}%
(s-V_{n}(s))\right| }{(s(1-s))^{1/2-\nu }}=O_{p}(n^{-\nu }),  \label{2.2}
\end{equation}
where for each $n\geq 1$, $U_{n}(s)=j/n$ for $U_{j,n}\leq s<U_{j+1,n}$ is the uniform empirical
$df$ and $V_{n}(s)=U_{j,n}$ for $(j-1)/n<s\leq j/n,$ and $%
V_{n}(0)=U_{1,n},$ is the uniform quantile function and, finalyy, $%
U_{1,n}\leq ...\leq U_{n,n\text{ }}$are the order statistics of $%
U_{1},...,U_{n}$ with by convention $U_{0,n}=0=1-U_{n+1,n}.$

\bigskip 
\noindent From now on, all the results are assumed to hold on this probability space and we therefore may use the general representation for the empirical $df$ $G_{n}$ based on $Y_{1},...,Y_{n}$ and for the order
statistics $Y_{1,n}\leq ...\leq Y_{n,n}$ by their uniform counterparts
\begin{equation}
\{1-G_{n}(x),0\leq x<+\infty ,n\geq 0\}=\{U_{n}(1-G(x)),0\leq x<+\infty
,n\geq 0\}  \label{2.3a}
\end{equation}
and 
\begin{equation}
\{Y_{j,n},1\leq j\leq n,n\geq 1\}=\{G^{-1}(1-U_{n\_j+1,n}),1\leq j\leq
n,n\geq 1\}.  \label{2.3b}
\end{equation}

\bigskip

\noindent We introduce these notation for $p\geq 1$,

\begin{equation*}
R(x,z,G)=R_{1}(x,z,G)=(1-G(x))^{-1}\int_{x}^{z}(1-G(t))dt,\ \ \ x<z\leq
y_{0},
\end{equation*}
\begin{equation*}
R_{p}(x,z,G)=(1-G(x))^{-1}\int_{x}^{z}dy_{1}\int_{y_{1}}^{z}...dy_{p-1}%
\int_{y_{p-1}}^{z}(1-G(t))dt,\text{ }x<z\leq y_{0}
\end{equation*}
with $R_{p}(x,z,G)\equiv R_{p}(x,G),$ $p\geq 1,$ and $x_{n}=G^{-1}(1-k/n) $ and $z_{n}=G^{-1}(1-\ell /n),$%
\begin{equation*}
\mu _{n}(k,\ell )=nk^{-1}\int_{x_{n}}^{z_{n}}(1-G(t))dt,\text{ }\tau (k,\ell
)=nk^{-1}\int_{x_{n}}^{z_{n}}dy\int_{y}^{z_{n}}(1-G(t))dt,
\end{equation*}
with $\mu _{n}(k,1)=\mu _{n}(k)$, $\tau (k,1)=\tau (k).$\\

\noindent If the $df$ is not specified in $R_{p}(\cdot )$, it is assumed that $R_{p}(\cdot
)=R_{p}(\cdot ,G)$ with $G(x)=F(e^{x})$. Finally put 
\begin{equation*}
N_{n}(0,k,\ell )=\left\{
(n/k)^{1/2}\int_{x_{n}}^{z_{n}}B_{n}(1-G(t))dt\right\} /R_{1}(x_{n}),
\end{equation*}
\begin{equation*}
N_{n}(2,\cdot )=-(n/\cdot )B_{n}(\cdot /n),\text{ }n\geq 1,
\end{equation*}
\begin{equation*}
N_{n}(3,k,\ell )=\left\{
(n/k)^{1/2}dy\int_{0}^{z_{n}}\int_{y}^{z_{n}}B_{n}(1-G(t))dt\right\}
/R_{2}(x_{n})
\end{equation*}
and 
\begin{equation*}
E_{n}(\ell )=nU_{\ell +1,n}/n.
\end{equation*}

\noindent We prove in the next sections that each $T_{n}(i)$ is asymptotically either one of
these $rv's$ or a linear combination of them. Their asymptotic laws are
described in :

\begin{theorem}[C] \label{theo3}
We have

\begin{enumerate}
\item  For each n, let 
$$
W_{n}(k,\ell)=(N_{n}(0,k,\ell ),N_{n}(3,k,\ell ),N_{n}(2,k))
$$ 

\noindent and 

$$
W_{n}(\ell,1)=(N_{n}(0,\ell ,\ell ),N_{n}(3,\ell ,1),N_{n}(2,\ell )).
$$

\noindent Then the vector $(W_{n}(k,\ell),W_{n}(\ell,1)$ is Gaussian. If further $\ell \rightarrow +\infty ,$ $k\rightarrow +\infty $, $k/n\rightarrow 0$, $\ell /k\rightarrow 0$\ \ \ as $n\rightarrow +\infty $
and $F\in \Gamma $, then $(W_{n}(k,\ell),W_{n}(\ell,1))$ converges in distribution to an $\mathbb{R}^{6}$-Gaussian random $rv$ $(W(1),W(2)$, where $W(1)$ and $W(2)$ are
independent vectors with the same covariance matrix :

\begin{equation*}
\left( 
\begin{array}{ccc}
3(\gamma +1)/(\gamma +2) &  &  \\ 
3(\gamma +1)(\gamma +3) & \frac{6(\gamma +1)(\gamma +2)}{(\gamma +3)(\gamma
+4)} &  \\ 
-1 & -1 & 1
\end{array}
\right) ,\text{ if }F\in D(\psi _{\gamma }),
\end{equation*}

\noindent and 

\begin{equation*}
\left( 
\begin{array}{ccc}
3 &  &  \\ 
3 & 6 &  \\ 
-1 & -1 & 1
\end{array}
\right) ,\text{ if }F\in D(\phi )\cup D(\Lambda )
\end{equation*}

\noindent where the symmetric matrices are only given in one side.

\item  Let be $\ell $ fixed, then for all $x\in R,$ $P(E_{n}(\ell )\leq x)$
converges to 
\begin{equation*}
P(E(\ell )\leq x)=\left\{ 
\begin{array}{c}
1-e^{-\ell x}\sum_{j=1}^{\ell }(j!)^{-1}(\ell x)^{j},\text{ }x\geq 0, \\ 
0\text{, elsewhere.}
\end{array}
\right.
\end{equation*}
\end{enumerate}
\end{theorem}

\bigskip

\noindent We do remark that the variance of $W(1)$ is obtained for $F\in D(\phi )\cup
D(\Lambda )$ by putting $\gamma =+\infty $ in the variance of W(1) for $F\in
D(\psi _{\gamma })$. This fact occurs almost always in this paper. Then,
throughout, for any expression depending on $\gamma ,$ $0<\gamma \leq \infty
,$ it is meant by $0<\gamma <\infty $ (resp. $\gamma =+\infty $), that $F\in
D(\psi _{\gamma })$ (resp. $F\in D(\phi )\cup D(\Lambda )$). Also all limits
with presence of $n$ are assumed to hold as $n\rightarrow +\infty .$

\begin{proof} Part 2 : See \cite{hall2}.\\

\noindent As to Part 1, $N_{n}(0,\cdot ,\cdot )$ and $N_{n}(3,\cdot ,\cdot )$ are
Reimann integrals so that any linear combination of the elements of $%
W_{n}(k)$ or of $W_{n}(\ell )$ is everywhere limits of sequences of $rv$ of
type of the following Riemann's sums : 
\begin{equation*}
\sum_{j=1}^{p}\sum_{i=1}^{q}\tau _{i,j}(p,q)B_{n}(t_{i,j}(p,q)),
\end{equation*}
as $(p,q)\rightarrow (+\infty ,+\infty ).$ But a Brownian bridge has
Gaussian finite-dimensional distributions. Thus, each finite linear
combination of the coordinates of $(W_{n}(k),W_{n}(\ell ))$ is a normal
$rv$ so that $(W_{n}(k),W_{n}(\ell ))$ is a Gaussian vector. That $W_{n}(k)$
and $W_{n}(\ell )$ are asymptotically independent follows from that $%
\lim_{n\rightarrow \infty }Cov(W_{n}(k),W_{n}(\ell ))$ is the $3\times 3$
null matrix by (\ref{7.2}) below. Formulas (\ref{3.34}), (\ref{4.3}), (\ref{4.36}), and (\ref{6.10}) 
and Lemmas \ref{l4.2} and \ref{l4.3} below yield the variance of $W(i)$, $i=1,2.$
\end{proof}

\section{Limit laws for T$_{n}(2,k,\ell )$ and T$_{n}(5).$} \label{sec3}

\subsection{Statements of the results.}

\begin{theorem} \label{t3.1}
Let $F\in \Gamma ,$ let $(k,\ell )$ be a couple of integers satisfying 
\begin{equation}
1<k=k(n)<n,\text{ }k\rightarrow \infty, \text{  } k/n\rightarrow 0,  \label{2.4a}
\end{equation}
\begin{equation}
1\leq \ell <k,\exists (\eta ,0\leq \eta <0.5),\ell k^{0.5-\eta }\rightarrow 0
\label{2.4b}
\end{equation}

\noindent (with $\eta =0$ for $F\in D(\phi )\cup D(\Lambda )$ and $\eta >0$ for $F\in
D(\psi _{\gamma })),$ then

\begin{equation*}
\mu _{n}(k,\ell )^{-1} \sqrt{k} \left\{ T_{n}(2,k,\ell )-\mu _{n}(\widetilde{k},\ell
)\right\} =N_{n}(0,k,\ell )+o_{p}(1)\rightarrow N(0,\sigma _{0}^{2}(\gamma
)),
\end{equation*}
where $\sigma _{0}^{2}(\gamma )=2(\gamma +1)/(\gamma +2),$ $0<\gamma \leq
+\infty $ and $\mu _{n}(\widetilde{k},\gamma )$ is obtained by replacing $%
x_{n}$ by $\widetilde{x}_{n}=G^{-1}(1-U_{k,n})$ in $\mu _{n}(k,\gamma ).$
\end{theorem}

\bigskip

\begin{remark} \label{r3.1} (\ref{2.4b}) automatically holds for $\ell$ fixed.
\end{remark}

\bigskip

\begin{corollary} \label{c3.1}
Under the same assumptions and notations used in Theorem \ref{theo3}, we have 
\begin{equation*}
\mu _{n}(k,\ell )^{-1} \sqrt{k} \left\{ T_{n}(2,k,\ell )-\mu _{n}(\widetilde{k})\right\}
=N_{n}(0,k,\ell )+o_{p}(1)\rightarrow N(0,\sigma _{0}^{2}(\gamma )),
\end{equation*}
where $\mu _{n}(\widetilde{k})=nk^{-1}\int_{\widetilde{x}%
_{n}}^{y_{0}}(1-G(t))dt\sim R_{1}(x_{n})\sim R_{1}(\widetilde{x}_{n})$ in
probability.
\end{corollary}

\bigskip

\begin{corollary} \label{c3.2}
Let $F\in D(\psi _{\gamma }),$ $0<\gamma \leq +\infty .$ If $(k,\ell )$
satisfies (\ref{2.4b}) with $\eta =0$ or less strongly, there exists $0< \zeta
>1/\gamma $ such that $\ell ^{1-\zeta +1/\gamma }/k^{0.5-\zeta
+1/\gamma }\rightarrow 0,$ then
$$
(y_{0}-x_{n})^{-1}k^{1/2}(T_{n}(2,k,\ell )-\mu _{n}(\widetilde{k%
}))\rightarrow N(0,2((\gamma +1)(\gamma +2)).
$$
\end{corollary}

\bigskip

\begin{remark}\label{r3.2}
Although these results hold on Cs\H{o}rg\"o and $al.$ (\cite{cchm})
probability space, the convergence in probability themselves are true
wathever the probability space is.
\end{remark}

\bigskip

\noindent Theorem \ref{t3.1} is already and partly proved by 
Cs\H{o}rg\"o and Mason(1985) (\cite{csorgo-mason}), Lo(1989) (\cite{gsloc})
for $F\in D(\phi )\cup D(\Lambda )$. But their proofs uses $%
u^{-1}\int_{1-u}^{1}(1-s)dG^{-1}(s)=r(u)$ instead of $R_{1}(x,G)$ for $%
1-G(x)=u$. Remark that $R_{1}(x)=r(u)$ for $u=1-G(x)$ if G is ultimately
continuous and increasing. We do not require at all such assumptions and
since the elements of this theorem are greatly used in all the remainder of
the paper, we should reprove it rigorously in a simultanuous treatement of
all our statistics.

\bigskip We now characterize the asymptotic normality of $T_{n}(2)$ when
attempting to replace $\mu _{n}(\widetilde{k})$ by $\mu _{n}(k).$ For this,
put 
\begin{equation}
\gamma _{n}(\cdot )=f(U_{\cdot +1,n})-f(\cdot /n),  \label{3.0a}
\end{equation}
\begin{equation}
T_{n}^{\ast }(2,k,\ell )=R_{1}(x_{n})^{-1}k^{1/2}(T_{n}(2,k,\ell )-\mu
_{n}(k))  \label{3.0b}
\end{equation}
and 
\begin{equation}
N_{n}(1,k,\ell )=N_{n}(0,k,\ell )+e_{1}(\gamma )N_{2}(k,\ell )  \label{3.0c}
\end{equation}

\noindent $e_{1}(\gamma )=(\gamma +1)/\gamma$, $e_{2}(\gamma )=\gamma+1$, $e_{2}(+\infty )=1$.

\begin{theorem} \label{t3.2}
Let $F\equiv (f,b)\in \Gamma $ and let $(k,\ell )$ satisfy (\ref{2.4a}) or (\ref{2.4b}). We have 
\begin{equation*}
T_{n}^{\ast }(2,k,\ell )=N_{n}(1,k,\ell )+e_{2}(\gamma )k^{1/2}\gamma
_{n}(k)+o_{p}(1),
\end{equation*}
with $N_{n}(1,k,\ell )\sim N(0,\sigma _{1}^{2}(\gamma )),$ $\sigma
_{1}^{2}(\gamma )=(\gamma ^{3}+\gamma ^{2}+2)/(\gamma ^{2}(\gamma +2)),$ $%
0<\gamma \leq +\infty$. Furthermore,

\begin{itemize}
\item[(a)]  $T_{n}^{\ast }(2,k,\ell )\rightarrow _{d}N(m_{0},\sigma
_{1}^{2}(\gamma ))$ iff $k^{1/2}\gamma _{n}(k)\equiv N_{n}(-\infty
)\rightarrow _{p}m_{0}/e_{2}(\gamma ).$

\item[(b)]  $T_{n}^{\ast }(2,k,\ell )\rightarrow _{d}N(m_{0},\sigma
_{0}^{2}(\gamma ))$ \ iff $\ \ N_{n}(-\infty )\rightarrow
_{p}N(m_{0}/e_{2}(\gamma ),(\sigma (-\infty )/e_{2}(\gamma ))^{2}$ with 
\begin{equation*}
\lim_{n\rightarrow +\infty }\sigma _{1}^{2}(\gamma )+\sigma ^{2}(-\infty
)+2e_{2}(\gamma )\text{ }Cov(N_{n}(1,k,\ell )N_{n}(-\infty ))=\sigma _{c}^{2}
\end{equation*}

\item[(c)]  $T_{n}^{\ast }(2,k,\ell )\rightarrow _{p}+\infty $ \ \ iff \ $%
N_{n}(\infty )\rightarrow _{p}+\infty .$
\end{itemize}
\end{theorem}

\bigskip

\noindent This characterization is very simple and general since one has 
\begin{equation*}
nk^{-1/2}(U_{k,n}-k/n)=N_{n}(2,k)+O_{p}(n^{-\nu }k^{-\nu })
\end{equation*}
by (\ref{2.2}) and hence $nU_{k,n}/k\rightarrow _{p}1$ as $k/n\rightarrow 0.$
Examples as in Hauesler and Teugels(1985) (\cite{haeusler-teugels}) may be treated in very simple ways
since in all their models $f(u)=f_{1}(u)(1+f_{2}(u))$ with $f_{1}^{\prime
}(u)\rightarrow 0$ as $u\rightarrow 0$ \ (for instance $f_{1}(u)=au^{\rho
},\rho >1),$ $f_{1}(u)=\exp (-b/u)$ and $f_{2}(u)\rightarrow 0).$ For all
these models, we have :

\begin{corollary}\label{c3.3}
Let $F\equiv (f,b)\in \Gamma $ and let $(k,\ell )$ satisfy (\ref{2.4a}) or (\ref{2.4b}). If
furthermore (\ref{1.7}), (\ref{1.8}) or (\ref{1.7}) holds then 
\begin{equation*}
T_{n}^{\ast }(2,k,\ell )=N_{n}(1,k,\ell
)+o_{p}(1)\rightarrow N(0,\sigma _{1}^{2}(\gamma )),
\end{equation*}
for $0<\gamma \leq \infty .$
\end{corollary}

\bigskip

\noindent One proves this corollary from Theorem \ref{t3.2} by using (\ref{2.2}). Let (\ref
{1.7}) be satisfied, thus 
\begin{equation*}
k^{1/2}\gamma _{n}(k)=nk^{-1/2}(U_{k,n}-k/n)(u_{n}f^{\prime
}(u_{n}))(1+o_{p}(1))
\end{equation*}
\begin{equation*}
=N_{2}(2,k)+o_{p}(1)+(u_{n}f^{\prime
}(u_{n}))(1+o_{p}(1)),
\end{equation*}

\noindent where $n\times u_{n}/k\rightarrow _{p}1.$ Now, if (\ref{1.8}) holds, 
\begin{equation*}
\left| k^{1/2}\gamma _{n}(k)\right| \leq k^{1/2}\max (\left| f(k/n)\right|
,\left| f(U_{k+1,n})\right| \rightarrow _{p}0.
\end{equation*}
If (\ref{1.8}) holds, thus 
\begin{equation*}
\left| k^{1/2}\gamma _{n}(k)\right| \leq o_{p}(1)+u_{n}f_{1}^{\prime
}(u_{n})\times N_{n}(2,k)
\end{equation*}
\begin{equation*}
+k^{1/2}\text{ }\left| f(k/n)\right| \text{ }\max
(\left| f(k/n)\right| ,\left| f_{2}(U_{k+1,n})\right| \rightarrow _{p}0,
\end{equation*}

\noindent by what is above.\\

\noindent We should remark that lim $uf^{\prime }(u)\rightarrow 0$ is a fairly general
condition since it holds whenever the limits exist. Also limiting laws of $%
T_{n}(5)=T_{n}(2,\ell ,1)$ are particular cases of results stated above. Now we are going to prove the results of this sections.

\subsection{Proofs.}

They largely use technical results in L\^{o} (\cite{gslod}). Use (\ref{2.1})
and (\ref{2.2}) to get 
\begin{equation*}
R_{1}(x_{n})^{-1} k^{1/2}(T_{n}(2,k,\ell )-\mu _{n}(\widetilde{k},\ell
))=N_{n}(0,k,\ell )
\end{equation*}
\begin{equation*}
+Z_{n1}(k)+(\ell /k)^{1/2}
Z_{n1}+k^{-1/2}Z_{n2}(\ell )+Z_{n3}, \label{3.1}
\end{equation*}
with 
\begin{equation*}
(k/n)^{1/2}Z_{n3}=R_{1}(x_{n})^{-1}\int_{x_{n}}^{z_{n}}(\alpha
_{n}(1-G(t))-B_{n}(1-G(t))dt,
\end{equation*}
where $\alpha _{n}(\cdot )$ (resp. $\beta _{n}(\cdot ))=n^{1/2}(U_{n}(s)-s)$
\ $($resp. $\ =n^{1/2}(s-V_{n}(s))$ and
\begin{equation*}
Z_{n1}(\cdot )=\mp \left\{ \frac{nk^{1/2}}{\cdot }\int_{G^{-1}(1-\cdot
/n)}^{G^{-1}(1-U_{.+1,n})}B_{n}(1-G(t))dt\right\} /R_{1}(x_{n}).
\end{equation*}

\noindent We shall treat each term into statements denotes (S1.3), (S2.3), etc. We
first have to prove :

\begin{equation}
Z_{n2}\rightarrow _{p}0\text{ when }\ell \rightarrow \infty .  \tag{S1.3}
\end{equation}
First, we have by (\ref{2.2}), 
\begin{equation}
n\ell ^{1/2}(U_{\ell +1,n}-\ell /n)=N_{n}(2,\ell )+o_{p}(1)\rightarrow
_{d}N(0,1)  \label{3.2}
\end{equation}
so that $nU_{\ell +1,n}\rightarrow _{0}1$ and hence 
\begin{equation}
\forall \lambda >1,\lim_{{n\rightarrow +\infty }}\mathbb{P}((\ell /\gamma
n)\leq U_{\ell +1,n}\leq (\lambda \ell /n))=1.  \label{3.3}
\end{equation}
For convenience, if (\ref{3.3}) holds, we say : \textit{for all $\lambda >1,$ one
has $(\ell /\gamma n)\leq U_{\ell +1,n}\leq (\lambda \ell /n)$ with
probability as Near One as Whished (PNOW)}, for large values of n. \ Hence 
\begin{equation*}
G(t)\leq 1-1/n\text{, }uniformly\text{ }for\text{ }\widetilde{x}_{n}\leq
t\leq \widetilde{z}_{n}, \label{3.4}
\end{equation*}
with (PNOW) as n is large. Secondly,

\begin{lemma} \label{l3.1}
(See Fact 5 in L(1990)) (\cite{gsloc}). Let $h(\cdot )$ be a bounded function on $(\alpha
,A),$ $\alpha >0,$ and G be any $df$. In the integrals below make sense as
improper ones, one has 
\begin{equation*}
\left| \int_{-\infty }^{G^{-1}(1-\alpha )}h(1-G(t))p(t)dt\right| \leq
\sup_{\alpha \leq s\leq 1}\left| h(s)\right| \int_{-\infty
}^{G^{-1}(1-\alpha )}\left| p(t)\right| dt.
\end{equation*}
\end{lemma}

\noindent Combining (\ref{2.1}) and (\ref{3.4}) and Lemma 1 in L\^o(1990) (\cite{gslod}) yields for some $\nu ,$ $%
0<\nu <1/4,$%
\begin{equation}
\left| Z_{n3}\right| \leq O_{P}(k^{-\nu })\left\{ (n/k)^{1/2-\nu }\int_{%
\widetilde{x}_{n}}^{\widetilde{Z}_{n}}(1-G(t))^{1/2-\nu }dt\right\}
/R_{1}(x_{n}).  \label{3.5}
\end{equation}

\noindent We need three lemmas at this step.

\begin{lemma} \label{l3.2}
If $F\in \Gamma ,$ then (1-G(G$^{-1}(1-u))/u\rightarrow 1$ $as$ $%
u\rightarrow 0.$
\end{lemma}

\begin{proof}
It follows from Lemma 5 in Lo(1990) (\cite{gslod}) that either G(G$^{-1}(1-u))$=1-u or $1-u$
lies on some constancy interval of $G^{-1},$ say $]G(x-),G(x)[$, so that G$%
^{-1}(1-u)=x$ and hence $1\geq 1-G(G^{-1}(1-u))/u\geq (1-G(x))/(1-G(x-)).$
Now, by Lemma 1 in L\^o(1990) (\cite{gslod}), either $G\in D(\psi _{\gamma }),0<\gamma <+\infty$ and
consequently KARARE holds, 
\begin{equation}
1-G(x)=c(x)(y_{0}-x)^{\gamma }\exp (\int_{1}^{1/(y_{0}-x)}t^{-1} p(t)dt,
  \label{3.6}
\end{equation}
$x<y_{0}=x_{O}(G)$, where c(x)$\rightarrow 1$ as $x\rightarrow y_{0}$ and $p(t)\rightarrow 0$ as 
$t\rightarrow \infty $, or $G\in D(\Lambda )$ and by de Haan-Balkema's
representation (cf. Smith(1987)) 
\begin{equation}
1-G(x)=c(x)\exp (\int_{1}^{1/(y_{0}-x)} \ell(t)^{-1}dt,  \label{3.7}
\end{equation}

\noindent $-\infty<w<y_{0}=x_{O}(G)$ where $c(x)\rightarrow 1$ as $x\rightarrow y_{0}$ and $\ell $ admits a
derivative $\ell ^{\prime }(x)\rightarrow 0$ as $x\rightarrow y_{0}.$ In
both cases, we readily see that $(1-G(x))/(1-G(x-))\rightarrow 1$ as $%
x\rightarrow y_{0}.$ This completes the proof.
\end{proof}

\begin{lemma} \label{l3.3}
Let G$_{r}(x)=1-(1-G(x))^{r},r>0$. Then
\begin{itemize}
\item[(i)] $G\in D(\psi _{\gamma })\Longrightarrow G_{r}\in D(\psi _{\gamma })$ with 
$y_{0}=x_{O}(G)=x_{O}(G_{r})$ and 
\begin{equation*}
R_{1}(x,z,G)/R_{1}(x,z,G_{r})\rightarrow (\gamma r+1)/(\gamma +1),
\end{equation*}
as $(1-G(z))/(1-G(x))\rightarrow 0$ as $x\rightarrow y_{0}$ and $z\rightarrow y_{0}$.

\item[(ii)] $G\in D(\Lambda )\Rightarrow G_{r}\in D(\Lambda )$ with $%
y_{0}=x_{O}(G)=x_{O}(G_{r})$ and 
\begin{equation*}
R_{1}(x,z,G)/R_{1}(x,z,G_{r})\rightarrow r^{-1},
\end{equation*}
as $(1-G(z))/(1-G(x))\rightarrow 0$ as $x\rightarrow y_{0}$ and $z\rightarrow y_{0}$.
\end{itemize}
\end{lemma}

\begin{proof} \textbf{Part (i)}. Let $G\in (\psi _{\gamma }).$ Thus (\ref{3.6}) holds for $G%
_{r} $ by putting $c_{r}(x)=c(x)^{r}$, $\gamma _{r}=r\gamma ,$ $%
p_{r}(t)=rp(t).$ Hence $G_{r}\in (\psi _{\gamma r})$ and $%
x_{0}(G)=x_{0}(G_{r})$. Further, by Formula 2.5.4 of de Haan(1970) or Lemma
\ref{l4.1} below, 

$$
R_{1}(x,G)/(y_{0}-x)\rightarrow (\gamma +1)^{-1}
$$

\noindent and 

\begin{equation}
R_{1}(x,G_{r})/(y_{0}-x)\rightarrow (\gamma r+1)^{-1},  \label{3.8}
\end{equation}

\noindent as $x\rightarrow y_{0}.$ By Lemmas 7 and 8 of L\^{o}(1990), 
\begin{equation}
R_{1}(x,z,G)/R_{1}(x,G)\rightarrow 1,R_{2}(x,z,G)/R_{2}(x,z,G)\rightarrow 1,
\label{3.9}
\end{equation}
as $(1-G(z))/(1-G(x))\rightarrow 0$ as $x\rightarrow y_{0}$ and z$%
\rightarrow y_{0},$ whenever $G\in D(\psi )\cup D(\Lambda )$ (hence,
whenever $F\in \Gamma ).$ Now (\ref{3.8}) and (\ref{3.9}) together prove
Part (i).\\

\noindent \textbf{Part (ii)} : Let $G\in D(\Lambda ).$ By (\ref{3.7}), $G_{r}\in D(\Lambda )$
with $x_{0}(G)=x_{0}(G_{r})$ and thus (cf. Lemma A in L\^{o}(1990)) for any $%
t\in R,$%
\begin{equation}
(1-G_{r}(x+tR_{1}(x,G_{r})))/(1-G_{r}(x))\rightarrow \exp (-t),  \label{3.10}
\end{equation}

\noindent as $x\rightarrow x_{0}(G_{r})=y_{0}$ and 
\begin{equation}
(1-G(x+tR_{1}(x,G)))/(1-G(x))\rightarrow \exp (-t),  \label{3.11}
\end{equation}

\noindent as $x\rightarrow x_{0}(G_{r})=y_{0}$. By combining (\ref{3.10}) and (\ref{3.11}) implies 
\begin{equation}
(1-G_{r}(x+tR_{1}(x,G_{r})))/(1-G(x))\rightarrow \exp (-t/r),  \label{3.12}
\end{equation}

\noindent as $x\rightarrow x_{0}(G_{r})=y_{0}$. If for a sequence $x_{n}\rightarrow x_{0},$ one has $%
R_{1}(x_{n},G)/R_{1}(x_{n},G_{r})\rightarrow \nu, \text{ } 0<\nu <r-\varepsilon $ $as$
$n\rightarrow \infty$. This and by (\ref{3.12}) would imply 
\begin{equation*}
\lim \inf_{n\rightarrow \infty
}(1-G(x_{n}+tR_{1}(x_{n},G)))/(1-G(x_{n}))\geq \exp (-t(r-\varepsilon
)/t)>\exp (-t),
\end{equation*}
which is impossible because of (\ref{3.10}). Hence 
\begin{equation*}
\lim \inf_{x\rightarrow y_{0}}R_{1}(x,G)/R_{1}(x,G_{r})\geq r.
\end{equation*}
Similarly, one gets that $\lim \sup_{x\rightarrow
y_{0}}R_{1}(x,G)/R_{1}(x,G_{r})\geq r.$ Finally $R_{1}(x,G)/R_{1}(x,G_{r})%
\rightarrow r$ $as$ $x\rightarrow y_{0},$ which combined with (\ref{3.9})
proves Part (ii).
\end{proof}

\bigskip

\begin{lemma} \label{l3.4}

\noindent (i) If $F\in D(\phi )\cup D(\Lambda ),$ then $R_{1}(G(1-u))$ is slowly
varying at zero (SVZ).

\noindent (ii) If $F\in D(\psi _{\gamma }),$ then $R_{1}(G(1-u))$ is regularly varying
at zero with exponent $\gamma ^{-1}$ ($\gamma ^{-1}-RVZ).$
\end{lemma}

\bigskip

\begin{proof}
Part (i) Let $F\in D(\phi )\cup D(\Lambda ).$ Thus $G\in D(\Lambda )$ by
Lemma 1 in L\^{o}(1990) (\cite{gslod}). Next Lemma \ref{l3.2} and lemma 2 in Lô \cite{gsloa}) yield for $%
\lambda >0,\lambda \neq 1,$%
\begin{equation}
\left\{ G^{-1}(1-\lambda u)-G^{-1}(1-u)\right\}
/R_{1}(G^{-1}(1-u))\rightarrow -\log \lambda ,  \label{3.13}
\end{equation}
as $u\rightarrow 0.$ From this Lemma 4 in L\^{o}(1989) implies 
\begin{equation}
\left\{ G^{-1}(1-\lambda u)-G^{-1}(1-u)\right\} /s(u)\rightarrow -\log
\lambda ,  \label{3.14}
\end{equation}
as $u\rightarrow 0,$ where $s(\cdot )$ is SVZ and defined as in (\ref{1.3b}).
Hence 
\begin{equation}
R_{1}(G^{-1}(1-u))\sim s(u),\text{ }as\text{ }u\rightarrow 0,  \label{3.15}
\end{equation}
so that $R_{1}(G^{-1}(1-u))$ is SVZ, which is Part (i).\\

\noindent Part (ii). This is easily derived by Formula 2.5.4 of de Haan(1970) (cf.
Lemma \ref{l4.1} below and the previous results).
\end{proof}

\bigskip

\noindent Now we return back to the proof of Theorem \ref{theo3}. By (\ref{3.5}) and Lemma \ref{l3.2},
\ref{3.3} and \ref{3.4}, we have 
\begin{equation}
\left| Z_{n3}\right| \leq O_{P}(k^{-\nu })R_{1}(x_{n},G_{(1/2)-\nu
})/R_{1}(x_{n},G)=O_{P}(k^{-\nu })=o_{P}(1),  \label{3.16}
\end{equation}
where we have taken (\ref{3.8}) into account. This proves $(S1.3)$.

\begin{equation}
Z_{n3}\rightarrow _{P}0,\text{ when }\ell \text{ is fixed.}  \tag{S2.3}
\end{equation}
We have 
\begin{equation}
\left| Z_{n3}\right| \leq Z_{n3}(1)+Z_{n3}(2),  \label{3.17}
\end{equation}
with 
\begin{equation}
Z_{n3}(1)=(\frac{n}{k})^{1/2}\left\{ \int_{\widetilde{x}_{n}}^{G^{-1}(1-\ell
/n)}\left| \alpha _{n}(1-G(t))-B_{n}(1-G(t))\right| dt\right\} /R_{1}(x_{n}),
\label{3.18}
\end{equation}
which is an $O_{P}(k^{-\nu })$ by the same arguments used in (\ref{3.5})
(one also has G(t)$\leq 1-1/n$), and 
\begin{equation}
\pm Z_{n3}(2)=(\frac{n}{k})^{1/2}\left\{ \int_{G^{-1}(1-\ell /n)}^{%
\widetilde{z}_{n}}\left| \alpha _{n}(1-G(t))-B_{n}(1-G(t))\right| dt\right\}
/R_{1}(x_{n}).  \label{3.19}
\end{equation}
By Theorem C, 
\begin{equation}
\lim_{\lambda \rightarrow \infty }\lim_{n\rightarrow \infty }\mathbb{P}(\ell /\lambda
n)\leq U_{\ell +1,n}\leq \lambda \ell /n)=1,  \label{3.20}
\end{equation}
which is quoted as $(\ell /\lambda n)\leq U_{\ell +1,n}\leq \lambda \ell /n)$
holds with PNOW for large values of $\ n$ and $\lambda .$ Thus, with PNOW
for lare values of $n$ and $\lambda ,$%
\begin{equation}
Z_{n3}(2)=(\frac{n}{k})^{1/2}\left\{ \int_{t_{n}(\lambda \ell )}^{t_{n}(\ell
/\lambda )}\left| \alpha _{n}(1-G(t))-B_{n}(1-G(t))\right| dt\right\}
/R_{1}(x_{n})\equiv Z_{n3}^{\ast }(2),  \label{3.21}
\end{equation}
where $t_{n}(\cdot )=G^{-1}(1-\cdot /n).$ We get 
\begin{equation}
\mathbb{E}Z_{n3}^{\ast }(2)\leq 4(\frac{n}{k})^{1/2}\left\{
\int_{t_{n}(\lambda \ell )}^{t_{n}(\ell /\lambda )}(1-G(t))^{1/2}dt\right\}
/R_{1}(x_{n})  \label{3.22}
\end{equation}
\begin{equation*}
\leq 4R_{1}(x_{n},G_{1/2})/R_{1}(x_{n},G)\{\int_{t_{n}(\lambda \ell
)}^{y_{0}}(1-G(t))^{1/2}dt/\int_{t_{n}(k)}^{y_{0}}(1-G(t))^{1/2}dt
\end{equation*}
\begin{equation*}
-\int_{t_{n}(\ell /\lambda
)}^{y_{0}}(1-G(t))^{1/2}/\int_{t_{n}(k)}^{y_{0}}(1-G(t))^{1/2}\}.
\end{equation*}
By Lemmas 7 and 8 in L\^{o}(1990), and Lemma \ref{l3.3} above both
terms in brackets tend to zero and then $\mathbb{E}Z_{n3}^{\ast
}(2)\rightarrow 0.$ By letting $\lambda \rightarrow \infty $ and by using
Markov inequality, one arrives at $Z_{n3}^{\ast }(2)\rightarrow _{P}0.$ This
with (\ref{3.17}) and (3.18) together prove (S2.3).

\begin{equation}
Z_{n1}\rightarrow _{P}0.  \tag{S3.3}
\end{equation}
We have as in (\ref{3.2}) 
\begin{equation}
\left| Z_{n1}\right| \leq (\frac{n}{k})^{1/2}\left\{ \int_{t_{n}(\lambda
k)}^{t_{n}(k/\lambda )}\left| B_{n}(1-G(t))\right| dt\right\}
/R_{1}(x_{n})\equiv Z_{n1}^{\ast },  \label{3.23}
\end{equation}
with PNOW as n is large. Furthermore 
\begin{equation}
\mathbb{E}Z_{n1}^{\ast }\leq 3(\frac{n}{k})^{1/2}\left\{ \int_{t_{n}(k\ell
)}^{t_{n}(k/\lambda )}(1-G(t))^{1/2}dt\right\} /R_{1}(x_{n})  \label{3.24}
\end{equation}
\begin{equation*}
\leq 3R_{1}(x_{n},G_{1/2})/R_{1}(x_{n},G)\{R_{1}(t_{n}(k/\lambda
),G_{1/2})(1+o(1)
\end{equation*}
\begin{equation*}
-R_{1}(t_{n}(k\lambda ),G_{1/2})(1+o(1)\}/R_{1}(x_{n}).
\end{equation*}
The term in the bracket tends to zero for $F\in D(\phi )\cup D(\Lambda )$
and to $\lambda ^{1/\gamma }-\lambda ^{-1/\gamma }$ for $F\in D(\psi
_{\gamma })$ by Lemmas \ref{l3.3} and \ref{l3.4}. Since $\lambda$ is arbitrary and
greater than one, on get $\mathbb{E}Z_{n1}^{\ast }\rightarrow 0$ by letting $%
\lambda \downarrow 1.$ Finally (S3.3) holds by Markov's inequality and (\ref
{3.23}). 
\begin{equation}
(\ell /k)^{1/2}Z_{n1}(\ell )\rightarrow _{P}0.  \tag{S4.3}
\end{equation}
This is proved exactly as for $\ell \rightarrow \infty .$ When $\ell $ is
fixed, one uses (\ref{3.20}) instead of (\ref{3.2}) and the proof of (S3.3)
is valid again.

\begin{equation}
\ell k^{-1/2}Z_{n2}\rightarrow _{P}0  \tag{S5.3}  
\end{equation}
Let $F\in D(\psi _{\gamma }).$ Thus $G\in D(\psi _{\gamma })$ and by Lemma
\ref{l4.1} below, Lemmas \ref{l3.2}, \ref{l3.3}, \ref{l3.4} and Theorem C, 

\begin{equation}
\ell k^{-1/2}Z_{n2} \label{3.25}
\end{equation}

\begin{equation}
=\ell k^{-1/2}\{(1+o_{P}(1))(y_{0}-z_{n})/(y_{0}-x_{n})-(1+o_{P}(1))(y_{0}-%
\widetilde{z}_{n})/(y_{0}-x_{n})\},  
\end{equation}

\noindent whenever $\lim \sup_{n\rightarrow \infty }\ell k^{-1/2}<\infty .$ In this
special case, one can choose $\ell $ satisfying 
\begin{equation}
\exists (0<\xi <1/\gamma ),\ell ^{1-\xi +1/\gamma }\rightarrow 0,
\label{3.26}
\end{equation}
since for all $0<\xi <1/\gamma ,$%
\begin{equation*}
\sup ((y_{0}-z_{n}),(y_{0}-\widetilde{z}_{n}))/(y_{0}-x_{n})\leq 2\sup
(E_{n}(\ell )^{1/\gamma },1)(\ell /k)^{-\xi +1/\gamma }.
\end{equation*}

\noindent Consider here the two cases :\\

\noindent (b) Let $F\in D(\phi _{\alpha }).$ Then $R_{1}(x,G)\rightarrow 1/\alpha $
(cf. for instance Lemma 1 in L\^{o}(1990)). Using Lemma \ref{l3.2} and Theorem C,
one gets 

\begin{equation*}
\left| \ell k^{-1/2}Z_{n2}(\ell )\right|
\end{equation*}

\begin{equation}
 =\alpha (1+o(1))\left|
R_{1}(z_{n})(1+o(1))-(1+o_{P}(1)R_{1}(z_{n})\right| (\ell k^{-1/2}),
\label{3.27}
\end{equation}

as $\ell \rightarrow \infty .$ Then 
\begin{equation}
\lim \sup_{n\rightarrow \infty }(\ell k^{-1/2})=0\Rightarrow (\ell
k^{-1/2}Z_{n2}(\ell )\rightarrow _{P}0).  \label{3.28}
\end{equation}
(c) Let $F\in D(\Lambda ).$ Thus $G\in D(\Lambda )$ and by Lemma \ref{l3.2},
\begin{equation}
\left| \ell k^{-1/2}Z_{n2}(\ell )\right| \leq \left|
(1+o(1))R_{1}(z_{n})-O_{P}(1)R_{1}(\widetilde{z}_{n})\right| /R_{1}(x_{n}).
\label{3.29}
\end{equation}
But by (\ref{3.3}) and (\ref{3.4}), one has 
\begin{equation}
(\ell k^{-1/2})R_{1}(x_{n})/R_{1}(z_{n})\leq 2(\ell ^{1-\varepsilon
}/k^{(1/2)-\varepsilon })R_{1}(z_{n})/R_{1}(\widetilde{z}_{n}).  \label{3.30}
\end{equation}
Since $R_{1}(z_{n})/R_{1}(\widetilde{z}_{n})\rightarrow _{P}1,$ one has 
\begin{equation}
(\ell k^{-1/2})R_{1}(x_{n})/R_{1}(z_{n})\leq 3(\ell ^{1-\varepsilon
}/k^{(1/2)-\varepsilon }),  \label{3.31}
\end{equation}
with PNOW as n is large. Putting $\eta =(1-(1-2\varepsilon )/(1-\varepsilon
))/2$ completes the proof of (S5.3) whenever (\ref{2.4b}) holds. It remains
the normal term $N(0,k,\ell ).$ We have

\begin{lemma} \label{l3.5}
Let $F \in \Gamma .$ If $\ \ell /k\rightarrow 0$, then $N_{n}(0,k,\ell )\sim 
\mathcal{N}(0,s_{n}^{2}(0)),$ where $s_{n}^{2}(0)\rightarrow \sigma
_{0}^{2}=2(\gamma +1)(\gamma +2),0<\gamma \leq +\infty .$
\end{lemma}

\begin{proof}
By Theorem C, $N_{n}(0,k,\ell )$ is a normal random variable with variance
(for $\xi (s)=1-G(s))$ 
\begin{equation}
s_{n}^{2}(0)=\left\{ (n/k)\int_{x_{n}}^{z_{n}}\int_{s}^{z_{n}}\left\{ \min
(\xi (s),\xi (t))-\xi (s)\xi (t)\right\} dsdt\right\} /R_{1}(x_{n}).
\label{3.32}
\end{equation}
Considering the cases $(s<t)$ and $(s\geq t)$ for the function in brackets
yields 
\begin{equation}
s_{n}^{2}(0)=\left\{
(2n/k)\int_{x_{n}}^{z_{n}}\int_{s}^{z_{n}}(1-G(t))G(s)dsdt\right\}
/R_{1}(x_{n})\sim 2R_{2}(x_{n},z_{n})/R_{1}^{2}(x_{n}).  \label{3.33}
\end{equation}
By Lemmas 7 and 8 in Lo(\cite{gsloc}) (see Formula (\ref{3.9}) above), Theorems
2.5.6 and 2.6.1 of de Haan(1970) (see Lemma \ref{l4.1} below) and Lemma \ref{l3.2} above, 
\begin{equation}
s_{n}^{2}(0)\sim 2R_{2}(x_{n},z_{n})/R_{1}^{2}(x_{n})\sim K(\gamma ),
\label{3.34}
\end{equation}
with $K(\gamma )=(\gamma +1)(\gamma +2),0<\gamma \leq +\infty .$
\end{proof}

\bigskip

\noindent From this, we have the followings facts.\\

\noindent (i) Theorem \ref{t3.1} is jointly proved by the statements (\ref{3.1}), (S3.1),
...,(S5.3 and by Lemma \ref{l3.5}.\\

\noindent (ii) Corollary \ref{c3.1} is obtained from Theorem \ref{t3.1} by showing that 
\begin{equation}
\left\{ (nk^{-1/2})\int_{z_{n}}^{y_{0}}(1-G(t))dt\right\} /R_{1}(x_{n})\sim
(\ell k^{-1/2})R_{1}(z_{n})/R_{1}(x_{n})\rightarrow 0.  \tag{S6.3}
\end{equation}

\bigskip

\noindent But, returning back to $(S5.3)$, we see that $(S6.3)$ holds for $F\in D(\phi )$ if and only if
$\ell k^{-1/2})\rightarrow 0$, for $F\in D(\psi _{\gamma })$ if and only if whenever \ref{3.26} holds and for $F\in D(\Lambda )$ whenever (\ref{2.4b}) holds.\\

\noindent (iii) Corollary \ref{c3.2} follows from Corollary \ref{c3.1} and by Formula 2.6.3 of de
Haan(1970) (See Lemma \ref{l4.1} below) and Formula (\ref{3.26}) above.\\

\noindent We have now to prove Theorem \ref{t3.2}. For this, we need 

\begin{lemma} \label{l3.6}
Let $F\equiv (f,b)\in \Gamma $ and $k$ satisfies (\ref{2.4a}), then 
\begin{equation*}
\xi _{n}(k)=k^{-1/2}(\widetilde{x}_{n}-x_{n})/R_{1}(x_{n})
\end{equation*}
\begin{equation*}
=e_{1}(\gamma )N_{n}(2,k)+e_{2}(\gamma )k^{-1/2} \gamma
_{n}(k)(1+o_{P}(1))+o_{P}(1).
\end{equation*}
\end{lemma}

\bigskip

\begin{proof}
\noindent (a) Let $F\in D(\Lambda ).$ Then (\ref{1.4}) and (\ref{3.15}) yield 
\begin{equation}
\xi _{n}(k)=k^{1/2}(1+o_{P}(1))\left\{ \left[ s(U_{k+1,n})/s(U_{k,n})\right]
-1\right\} +k^{1/2}\int_{k/n}^{U_{k+1,n}}t^{-1}s(t)dt.  \label{3.35}
\end{equation}

\noindent By (\ref{1.3b}) and (\ref{3.2}) $sup$ $(resp.$ $\ inf)$ $\{s(t)/s(k/n),t\in
I_{n}\}\rightarrow _{P}1,$ where $T_{n}$ is the closed and random interval
formed by $k/n$ and $U_{k+1,n}.$ Thus 
\begin{equation}
k^{1/2}s(k/n)^{-1}\int_{k/n}^{U_{k+1,n}}t^{-1}s(t)dt=(1+o_{P}(1)) \left\{ nk^{-1/2}(U_{k+1,n}-k/n)\right\} + o_{P}(1). \label{3.36}
\end{equation}
Also,

\begin{equation}
k^{1/2}\left\{ s\left( U_{k+1,n}\right) /\left( k/n\right) -1\right\}
=k^{1/2}\gamma _{n}\left( k\right) \left( 1+o_{p}\left( 1\right) \right)
=  \label{3.37}
\end{equation}

\begin{equation}
=o_{p}\left( nk^{-1/2}\left( U_{k+1,n}-k/n\right) \right).  \label{3.37}
\end{equation}

\noindent See (\ref{3.38}) - (\ref{3.40}) for more details.\\

\noindent (b) Let $F\in D\left( \phi \right) $. This case is exactly the preceeding
since $G\in D\left( \Lambda \right) $ and (\ref{1.4}) holds.\\

\noindent (c) Let $F\in D\left( \psi _{\gamma }\right) $. Use $R_{1}\left(
x_{n}\right) \sim \left( Y_{0}-x_{n}\right) /\left( \gamma +1\right) $ and
get 

\begin{equation}
\xi_{n}(k) \sim k^{1/2} \left\{ ( y_{0}-{\widetilde{x}}_{n} ) - ( y_{0}-x_{n}) \right\} / \{( y_{0}-x_{n})/(\gamma+1)\} \label{3.38}
\end{equation}

\begin{equation}
\sim \times (\{(y_{0}-{\widetilde{x}}_{n})/(y_{0}-x_{n})\} -1)
\end{equation}

\noindent Now, by (\ref{2.6}), 

\begin{equation}
\left( 1+\gamma_{n}\right) \left(\frac{k}{n} U_{k+1,n} \right) ^{-\varepsilon _{n}+1/\gamma }-1
\leq \frac{ y_{0}-\widetilde{x}_{n}}{y_{0}-x_{n}}-1
\leq \left( 1+\gamma _{n}\right) \left( \frac{k}{n} U_{k+1,n} \right)^{\varepsilon _{n}+1/\gamma}-1,  \label{3.39}
\end{equation}

\noindent where $\varepsilon _{n}=\sup \left\{ b\left( t\right) ,\text{ }t\leq \max
\left( U_{k+1,n},\text{ }k/n\right) \right\} $ ${\rightarrow}_{P}0$. Since $%
nU_{k+1,n}/k$ ${\rightarrow}_{P}1$, one gets

\begin{equation*}
\xi _{n}\left( k\right) =\left( \gamma +1\right) k^{1/2}\gamma _{n}\left(
k\right) \left( 1+o_{p}\left( 1\right) \right)
\end{equation*}

\begin{equation}
+\gamma ^{-1}\left( \gamma
+1\right) nk^{-1/2}\left( U_{k+1,n}-k/n\right) \left( 1+o_{p}\left( 1\right)
\right) .  \label{3.40}
\end{equation}

\bigskip

\noindent Now (\ref{2.2}), (\ref{3.36}), (\ref{3.40}) and point (b) just below
complete the proof of Lemma \ref{l3.6}.\\

\noindent We return to the proof of Theorem \ref{t3.2}. By (\ref{3.1}), (S1.3), ..., (S6.3), (\ref{3.40}) 
\begin{equation*}
T_{n}^{\ast}(2,k\ell) = N_{n}\left( 0,k,\ell \right) +k^{1/2}\left\{\frac{k}{n}\int_{x_{n}^{\sim }}^{x_{n}}1-G\left( t\right) \text{ }dt\right\} /\text{
}R_{1}\left( x_{n}\right) +o_{p}\left( 1\right) .
\end{equation*}
Using Lemmas \ref{l3.2} and \ref{l3.6} gives 

\begin{equation}
T_{n}^{\ast }\left( 2,k,\ell \right) =N_{n}\left( 0,k,\ell \right)
+e_{1}\left( \gamma \right) N_{n}\left( 2,k\right) +e_{2}\left( \gamma
\right) k^{1/2}\gamma _{n}\left( k\right) \left( 1+o_{p}\left( 1\right)
\right) +o_{p}\left( 1\right).  \label{3.42}
\end{equation}

\noindent From this, we easily conclude and get the characterization of Theorem \ref{t3.2}. To compute $\mathbb{E}\left( N_{n}\left(
0,k,\ell \right) N_{n}\left( 2,k\right) \right) ,$ we recall that $E\left(
B\left( s\right) B\left( t\right) \right) =\min \left( s,t\right) -st$ if $%
\left\{ B\left( t\right) ,0\leq t\leq 1\right\} $ is a Brownian bridge, and
an easy calculation yields 
\begin{equation}
\mathbb{E}\left( N_{n}\left( 0,k,\ell \right) N_{n}(2,\ell)\right) \rightarrow -1.
\label{3.43}
\end{equation}

\noindent This and Lemma \ref{l3.5} suffice to compute $\sigma _{1}\left( \gamma \right) $.
All the proofs are now complete.

\end{proof}

\bigskip

\section{Limit laws for $A_{n}\left( 1,k,\ell \right)$} \label{sec4}

\noindent We need some generalized forms of Lemmas due to de Haan (1970) or to L\^{o}
(1990).

\bigskip

\begin{lemma} \label{l4.1}
Let $F\in \Gamma $, then for any integer $p\geq 1$,

\noindent (i) $R_{p}\left( x,F\right) \sim \left( x_{o}-x\right) ^{p}\left\{
\prod_{j=1}^{j=p}\left( \gamma +j\right) ^{-1}\right\} ,$ as $s\rightarrow
x_{o}$ whenever $F\in D\left( \psi _{\gamma }\right) $,\\

\noindent and\\

\noindent (ii) $R_{p}\left( x,G\right) \sim R_{1}\left( x,G\right) ^{p}$, as $%
x\rightarrow y_{o}$, whenever $F\in D\left( \Lambda \right) \cup D\left(
\phi \right) $.
\end{lemma}

\bigskip

\begin{proof} of Lemma \ref{l4.1}. \textbf{Part (i)} is obtained by routine computations from (\ref{1.2}).\\

\noindent \textbf{Part (ii)} is easily proved from Lemma 2.5.1 and Theorem 2.5.2b of de Haan (1990)
showing that
$$
G_{1}=G\in D\left( \Lambda \right) \Rightarrow G_{2}\left( \circ \right)
=1-\int_{\circ }^{y_{o}}1-G_{1}\left( t\right) dt\in D\left( \Lambda
\right)
$$

\noindent  and $R_{1}\left( x,\widetilde{G}_{1}\right) \sim R_{1}\left(
x,\widetilde{G}_{2}\right) $ as $x\rightarrow x_{o}\left( G\right)$. By applying this $p$ times gives (ii).
\end{proof}

\bigskip

\begin{lemma} \label{l4.2}. Let $F\in \Gamma $. Then for any $p\geq 1,$

\noindent (i) $\left( z-x\right) ^{p}/$ $R_{p}\left( z,G\right) \rightarrow +\infty ,$

\noindent (ii) $R_{p}\left( x,z,G\right) $ $/$ $R_{p}\left( x,G\right) \rightarrow 1,$

as $x\rightarrow x_{o}\left( G\right) ,$ $z\rightarrow x_{o}\left( G\right)
, $ $\left( 1-G\left( z\right) \right) $ $/$ $\left( 1-G\left( x\right)
\right) \rightarrow 0.$
\end{lemma}

\bigskip

\begin{proof} \textbf{Part (i)} is easily derived from Part (ii) of Lemma \ref{l4.1} above and Lemma 8 of
L\^{o} (1990). To \textbf{prove (ii)}, put 
\begin{equation}
m\left( a,b,x\right)
=\int_{x}^{z}\int_{x_{1}}^{z}...\int_{x_{a-1}}^{z}\int_{y_{1}}^{y_{o}}...%
\int_{y_{b-1}}^{y_{o}}1-G\left( t\right) \text{ }dt\text{ }d\vartriangle
_{a,b},  \label{4.1}
\end{equation}
where $d\vartriangle _{a,b}=dx_{1}...$ $dx_{a-1}\times dy_{1}...$ $%
dy_{b-1}=d\vartriangle _{a}\times d\vartriangle _{b}$ and 
\begin{equation*}
m\left( 0,b,x\right)
=\int_{x}^{y_{o}}\int_{y_{1}}^{y_{o}}...\int_{y_{b-1}}^{y_{o}}1-G\left(
t\right) \text{ }dt\text{ }d\vartriangle _{b},
\end{equation*}
and 
\begin{equation*}
m\left( a,0,x\right)
=\int_{x}^{z}\int_{x_{1}}^{z}...\int_{x_{a-1}}^{z}1-G\left( t\right) \text{ }%
dt\text{ }d\vartriangle _{a}.
\end{equation*}
Straightforward manipulations yield for any $p\geq 2$, for $z<x<y_{o}$, 
\begin{equation}
R_{p}\left( x,z\right) =R_{p}\left( x\right) \left( \left( m\left(
0,p,x\right) \text{ }/\text{ }m\left( 0,p,x\right) \right) -\sum_{j=1}^{j=p}%
\frac{\left( z-x\right) ^{p-j}}{\left( p-j\right) !}\times \frac{m\left(
0,j,z\right) }{m\left( 0,p,x\right) }\right)  \label{4.2}
\end{equation}
As in Formula $(2.10)$ in L\^{o} (1990), 
\begin{equation}
0\leq m\left( 0,p,z\right) \text{ }/\text{ }m\left( 0,p,x\right) \leq \left(
1+\left( z-x\right) ^{p}\text{ }/\text{ }R_{p}\left( z\right) \right) ^{-1},
\end{equation}
and for any $j,$ $1\leq j\leq p$, 
\begin{equation}
\frac{\left( z-x\right) ^{p-j}m\left( 0,j,z\right) }{m\left( 0,p,x\right) }=%
\frac{R_{j}\left( z\right) }{\left( z-x\right) ^{j}}\times \frac{\left(
z-x\right) ^{p}}{R_{p}\left( z\right) }\times \frac{m\left( 0,p,z\right) }{%
m\left( 0,p,x\right) }  \label{4.4}
\end{equation}
since $\sup_{x\geq 0}$ $x\left( 1+x\right) ^{-1}=1$. (\ref{4.2}), (\ref{4.3}%
) and (\ref{4.4}) part (i) together ensure part (ii).
\end{proof}

\noindent Here are our results for $A_{n}\left( 1,k,\ell \right) $.

\bigskip

\begin{theorem} \label{t4.1}. Let $F\in \Gamma $. Suppose that (\ref{2.4a}) or (\ref{2.4b}) holds, then 
\begin{equation*}
R_{2}\left( x_{n}\right) ^{-1}k^{1/2}\left( A_{n}\left( 1,k,\ell \right)
-\tau \left( (\widetilde{k},\ell \right) \right) =N_{n}\left( 3,k,\ell \right)
+o_{p}\left( 1\right) \rightarrow ^{d}N\left( 0,\sigma _{2}^{2}\left( \gamma
\right) \right) ,
\end{equation*}

\noindent where $\tau \left( \widetilde{k},\ell \right) =nk^{-1}\int_{\widetilde{x}_{n}}^{z_{n}}\int_{y}^{z_{n}}1-G\left( t\right) $ $dt\sim R_{2}\left(
x_{n},z_{n}\right) \sim R_{2}\left( x_{n}\right)$ in probability and
$$
\sigma _{2}^{2}\left( \gamma \right) =6\left( \gamma +1\right) \left( \gamma
+2\right) \left\{ \left( \gamma +3\right) \left( \gamma +4\right)
\right\} , \text{ } 0<\gamma \leq +\infty. 
$$.
\end{theorem}

\begin{remark} \label{r4.1}.  If $F\in D\left( \phi _{\alpha }\right), \alpha>0$, $R_{2}\left(
x_{n}\right) $ and $\tau \left( \widetilde{k},\ell \right) $ tend to $\alpha
^{-2}$ so that $A_{n}\left( 1,k,\ell \right) ^{-1/2}\rightarrow _{p}\alpha$, as $\left(\ell k^{1/2-\eta },k/n\right) \rightarrow \left( 0,0\right) $. We then conclude that $A_{n}\left(1,k,\ell \right)^{1/2}$ is an asymptotically consistent estimator of $\alpha$. It will be studied elsewhere.
\end{remark}

\begin{remark} \label{r4.2}. One can aweaken the assumptions on $\left( k,\ell \right) $ by replacing
(\ref{2.4b}) by $\lim_{n\rightarrow +\infty }\sup \ell k^{-1/2}<\infty $ or $%
\ell ^{1+1/\gamma }=o\left( k^{1/2-\eta +1/\gamma }\right) $ for $F\in
D\left( \psi _{\gamma }\right) ,$ by $\ell k^{-1/2}\log k\rightarrow 0$ for $%
F\in D\left( \phi \right) $ and by $\ell =o\left( k^{1/2-\eta }\left( \log
k\right) ^{2\eta +1}\right) $ for $F\in D\left( \Lambda \right)$. (Cf. the lines
following (\ref{4.18})). Also, $\tau \left( \widetilde{k}\right) $ can replace $%
\tau \left( \widetilde{k},\ell \right) $ in Theorem \ref{t4.1} under the same
assumptions.
\end{remark}

\noindent Now, put for $0<\gamma \leq +\infty $,

$$v_{n}\left( 0\right) \left\{ A_{n}\left( 1,k,\ell \right) -c_{n}\left(
0\right) \right\} =A_{n}^{\ast }\left( 1,k,\ell \right),
$$
$$
v_{n}\left(0\right) =k^{1/2}/R_{2}\left( x_{n}\right),  c_{n}\left( 0\right)=\tau \left( k\right),
$$
$$
e_{3}\left( \gamma \right) =\gamma ^{-1}\left( \gamma +2\right),
$$
$$
\kappa \left( \gamma \right) =\left( \gamma +1\right) / \left(\gamma +2\right),
$$ 
$$
e_{4}\left( \gamma \right)=\gamma +2,0<\gamma <+\infty, \text{ }e_{4}\left( \infty \right) =1
$$

\noindent and
$$
\sigma _{3}^{2}\left( \gamma \right) =\left( 5\gamma ^{4}+11\gamma
^{3}+4\gamma ^{2}+7\gamma +12\right) \left\{ \gamma ^{2}\left( \gamma
+3\right) \left( \gamma +4\right) \right\}
$$ 

\bigskip

\begin{theorem} \label{t4.2} Suppose that the assumptions of Theorem \ref{t4.2} are satisfied. We have

$$
A_{n}^{\ast }\left( 1,k,\ell \right) =N_{n}\left( 4,k,\ell \right)
+e_{4}\left( \gamma \right) N_{n}\left( -\infty \right) \left( 1+o_{p}\left(
1\right) \right)
$$

$$
+\left( 2\kappa \left( \gamma \right) \right)^{-1} C_{n} \times \left( 1+o_{p}\left( 1\right) +o_{p}\left( 1\right) \right)
$$

\noindent where

$$
C_{n}= e_{1}\left( \gamma \right) ^{2}N_{n}\left( 2,k\right) ^{2}k^{-1/2}+2k^{-1/2}e_{1}\left( \gamma \right) e_{2}\left( \gamma \right) N_{n}\left(2,k\right) N_{n}\left( -\infty \right) \left( 1+o_{p}\left( 1\right) \right)
$$

$$
+e_{1}\left( \gamma \right) ^{2}N_{n}\left( -\infty \right)^{2}k^{-1/2}\left( 1+o_{p}\left( 1\right) \right) 
$$

\bigskip
\noindent where $N_{n}\left( 4,k,\ell \right) \sim N\left( 0,\sigma _{3}^{2}\left(
\gamma \right) \right) $. Then,\\

\begin{itemize}
\item[(i)] $A_{n}^{\ast }\left( 1,k,\ell \right) \rightarrow ^{d}N\left( m,\sigma
_{3}^{2}\left( \gamma \right) \right) $ iff $N_{n}\left( -\infty \right)
\rightarrow _{p}m$ $/$ $e_{4}\left( \gamma \right)$.

\item[(ii)] $A_{n}^{\ast }\left( 4,k,\ell \right) \rightarrow ^{d}N\left( m,\sigma
_{c}^{2}iff\ N_{n}\left( -\infty \right) \right) \rightarrow ^{d}N\left(
m/e_{4}\left( \gamma \right) ,\sigma \left( -\infty \right) ^{2}\right)$,\\
with
$\lim_{n\rightarrow +\infty }\sigma _{3}^{2}\left( \gamma \right) +\sigma
\left( -\infty \right) ^{2}+2e_{4}\left( \gamma \right) CovN_{n}\left(
4,k,\ell \right) N_{n}\left( -\infty \right) =\sigma _{c}^{2}.$
\item[(iii)] $A_{n}^{\ast }\left( 4,k,\ell \right) \rightarrow _{p}+\infty $ $%
\left( resp.\text{ }-\infty \right) $ iff $N_{n}\left( -\infty \right)
\rightarrow _{p}+\infty $ $\left( resp.\text{ }-\infty \right)$.
\end{itemize}
\end{theorem}

\bigskip

\begin{remark} \label{r4.3} The characterizations are identical in Theorems \ref{t3.2} and \ref{t4.2}. We pointed
out in section \ref{sec3} that we have the asymptotic normality whenever (\ref{1.7}%
) holds. The following examples concentrate on the case where $uf^{\prime
}\left( u\right) $ has not a limit as $u\rightarrow 0$.
\end{remark}

\begin{example} Let $f\left( u\right) =u$ $\sin \left( 1/u\right) .$

\noindent $f^{\prime }$ exists and $uf^{\prime }\left( u\right) =u$ $\sin \left(
1/u\right) -\cos \left( 1/u\right) $ does not converge as $u\rightarrow 0$. But if there exists a sequence of integers $\left( p_{n}\right) _{n\geq 1}$
such that $\left( 1/u_{n}\right) -2\pi p_{n}\rightarrow b,$ $-\pi <b\leq \pi 
$, then $u_{n}f^{\prime }\left( u_{n}\right) \rightarrow -\cos b=a$.
Returning to the proof of Corollary \ref{c3.2} and remembering that $\left|
U_{k+1,n}-k/n\right| \leq n^{-1/4}$, a.s. one has for $0<2\alpha <1/4$ :\\

\noindent If $p_{n}=\left[ n^{\alpha }\right] ,$ $k_{n}=\left[ n/2\left( 2\pi
p_{n}+b\right) \right] ,$ $0\leq b\leq 2\pi $ and $\cos b=a$, then\\

\noindent $T_{n}^{\ast }\left( 2,k,\ell \right) =N_{n}\left( 0,k,\ell \right) +\left(
e_{1}\left( \gamma \right) -ae_{2}\left( \gamma \right) \right) N_{n}\left(
2,k\right) +o_{p}\left( 1\right)$\\

\noindent and\\

$A_{n}^{\ast }\left( 2,k,\ell \right) =N_{n}\left( 3,k,\ell \right) +\left(
e_{3}\left( \gamma \right) -ae_{4}\left( \gamma \right) \right) N_{n}\left(
2,k\right) +o_{p}\left( 1\right)$.\\

\noindent In particular, if $b=0$ , then $T_{n}^{\ast }\left( 2,k,\ell \right) $
(resp. $A_{n}^{\ast }\left( 1,k,\ell \right) $) $\rightarrow ^{d}N\left(
0,\sigma _{o}^{2}\left( \gamma \right) \right) $ (resp. $N\left( 0,\sigma
_{2}^{2}\left( \gamma \right) \right) $ for $F\in D\left( \Lambda \right)
\cup D\left( \phi \right) $. This fact occurs for $F\in D\left( \psi
_{\gamma }\right) $ iff $\gamma \leq 1$ with $\cos b=1/\gamma .$
\end{example}

\bigskip

\textbf{Proof}. We proceed as for $T_{n}\left( 2,k,\ell \right) $ by general statements 
(S1.4), (S2.4), (S3.4), etc... First, use (\ref{2.1}), (\ref{2.3a}) and (\ref{2.3b}) to
obtain :

\begin{equation*}
R_{2}\left( x_{n}\right) ^{-1}k^{1/2}A_{n}\left( 1,k,\ell \right) -\tau
\left( \widetilde{k},\ell \right)  =N_{n}\left( 3,k,\ell \right)
+Q_{n1}\left( k\right) +\left( \ell /k\right) ^{1/2}  
\end{equation*}

\begin{equation}
+Q_{n1}\left( \ell \right) +\ell k^{-1/2}Q_{n2}\left( \ell \right) +Q_{n3}, \label{4.5}
\end{equation}

\bigskip

\noindent where\\

\bigskip 
\small

$$
Q_{n1}\left( k\right) =\left( n/k\right)^{1/2} \left( \left\{ \int_{\widetilde{x}_{n}}^{\widetilde{z}_{n}}   \int_{y}^{\widetilde{z}_{n}} B_{n}\left( 1-G\left(t\right) \right) \text{ }dtdy\right\} -\left\{ \int_{y}^{\widetilde{z}_{n}} B_{n} \left(1-G(t)) \right) 
\text{ }dtdy \right\} \right) /R_{2}( x_{n},
$$

\bigskip

$$
Q_{n1}\left( \ell \right) =\left( n/\ell \right) ^{1/2}\left( \left\{
\int_{x_{n}}^{\widetilde{z}_{n}}\int_{y}^{\widetilde{z}_{n}}B_{n}\left(
1-G\left( t\right) \right) \text{ }dtdy\right\} -\left\{ \int_{x_{n}}^{%
\widetilde{z}_{n}}\int_{y}^{\widetilde{z}_{n}}B_{n}\left( 1-G\left( t\right)
\right) \text{ }dtdy\right\} \right) /R_{2}\left( x_{n}\right),
$$

\bigskip

$$
Q_{n2}\left( \ell \right) =\left( n/\ell \right) ^{1/2}\left( \left\{
\int_{x_{n}}^{\widetilde{z}_{n}}\int_{y}^{\widetilde{z}_{n}}B_{n}
(1-G(t)){ }dtdy\right\} -\left\{ \int_{\widetilde{x}%
_{n}}^{z_{n}}\int_{y}^{z_{n}}1-G\left( t\right) \text{ }dtdy\right\} \right) 
/R_{2}\left( x_{n}\right)$$,

\noindent and\\

$$
Q_{n3}=\left( n/k\right) ^{1/2}\left( \int_{\widetilde{x}_{n}}^{\widetilde{z%
}_{n}}\int_{y}^{\widetilde{z}_{n}}\left\{ \alpha _{n}\left( 1-G\left(
t\right) \right) -B_{n}\left( 1-G\left( t\right) \right) \right\} \text{ }%
dtdy\right) /R_{2}\left( x_{n}\right).
$$

\Large
\bigskip
\noindent We show that each of these error terms tends to zero in probability.

\begin{equation}
Q_{n3}\rightarrow _{p}0.  \tag{S1.4}
\end{equation}

\noindent If $\ell \rightarrow \infty $, we get, as in (\ref{3.5}), for some $\nu 
$, $0<\nu <1/4$, 

\begin{equation}
\left| Q_{n3}\right| \leq 0_{p}\left( k^{-\nu}\right) R_{2}\left( 
\widetilde{x}_{n},\widetilde{z}_{n},G_{1/2-\upsilon }\right) \text{ }/\text{ 
}R_{2}\left( x_{n}\right) \rightarrow _{p}0\text{,}  \label{4.6}
\end{equation}

\noindent by Lemmas \ref{l3.2} and \ref{l4.2}. Now, let $\ell $ be fixed. Thus 

\begin{equation*}
Q_{n3} = \left( \frac{n}{k}\right) ^{1/2}\left\{ \int_{\widetilde{x}%
_{n}}^{z_{n}}\int_{y}^{z_{n}}\circ +\left( n/k\right) ^{1/2}\int_{\widetilde{%
x}_{n}}^{t_{n}\left( \ell \right) }\int_{z_{n}}^{\widetilde{z}_{n}}\circ
-\left( n/k\right) ^{1/2}\int_{z_{n}}^{\widetilde{z}_{n}}\int_{y}^{%
\widetilde{z}_{n}}.\right\} \text{ }/\text{ }R_{2}\left( x_{n}\right)
\end{equation*}

\begin{equation}
= :Q_{n3}\left( 1\right) +Q_{n3}\left( 2\right) +Q_{n3}\left( 3\right), \label{4.7}
\end{equation}

\noindent where $\circ$ stands for $\alpha _{n}\left( 1-G\left( t\right) \right)
-B_{n}\left( 1-G\left( t\right) \right) $ $dt$ $dy$. One quickly obtains for some $\nu$, $0<\upsilon <1/4$, 

\begin{equation}
Q_{n3}\left( 1\right) =O_{p}\left( k^{-\upsilon }\right) \rightarrow _{p}0,%
\text{ }as\text{ }k\rightarrow +\infty .  \label{4.8}
\end{equation}

\noindent Next, by Lemma \ref{l3.2}, with PNOW as $n$ and $\lambda $ are large, $\left| Q_{n3}\left( 3\right) \right|$ is less than

\begin{equation*}
 \leq \left\{ \left( n/k\right)
^{1/2}\int_{z_{n}}^{t_{n}\left( \ell /\lambda \right) }\int_{y}^{t_{n}\left(
\ell /\lambda \right) }\left| \alpha _{n}\left( 1-G\left( t\right) \right)
-B_{n}\left( 1-G\left( t\right) \right) \right| \text{ }dtdy\right\} \text{ }%
/\text{ }R_{2}\left( x_{n}\right) 
\end{equation*}

\begin{equation}
\leq \text{ }:Q_{n3}^{\ast }\left(3\right) ,  \label{4.9}
\end{equation}

\noindent with by Lemma \ref{l3.2}, 

\begin{equation*}
E\text{ }Q_{n3}^{\ast }\left( 3\right) \leq 4\left\{ \left( n/k\right)
^{1/2}\int_{z_{n}}^{t_{n}\left( \ell /\lambda \right) }\int_{y}^{t_{n}\left(
\ell /\lambda \right) }\left( 1-G\left( t\right) \right) ^{1/2}dt\text{ }%
dy\right\} /R_{2}\left( x_{n}\right) 
\end{equation*}

\begin{equation}
\leq 4\ell ^{1/2}\frac{\left(
t_{n}\left( \ell /\lambda \right) -t_{n}\left( \ell \right) \right) ^{2}}{%
R_{2}\left( z_{n}\right) }\times \left( k^{1/2}R_{2}\left( x_{n}\right)
/R_{2}\left( z_{n}\right) \right) ^{-1}.  \label{4.10}
\end{equation}

\noindent By Lemma 1 in L\^{o} (1990) and Lemmas 4.1.2 and Formulas (\ref{3.13}) and (%
\ref{3.15}) above, $F\in D\left( \phi _{\gamma }\right) $ implies that $%
R_{2}\left( x_{n}\right) \rightarrow \gamma ^{-2},R_{2}\left( x_{n}\right)
\rightarrow \gamma ^{-2},\left( t_{n}\left( \ell /\gamma \right)
-t_{n}\left( \ell \right) \right) $ $/$ $R_{2}\left( z_{n}\right)
\rightarrow \left( \log \lambda \right) ^{2}$ so that $E$ $Q_{n3}\left(
3\right) \rightarrow 0$. By the same arguments $EQ_{n3}^{\ast }\left(
3\right) \rightarrow 0$ when $F\in D\left( \Lambda \right)$ since

$$
\rho _{n}^{-1}\left( 0\right) =k^{1/2}R_{2}\left( x_{n}\right)/%
R_{2}\left( z_{n}\right) \geq k^{1/2-\varepsilon }\ell ^{\varepsilon }/2,
$$

\noindent for any $\varepsilon $, $0<\varepsilon <1/2$, as $n$ is large enough (use
SVZ functions properties).\\

\noindent For $F\in D\left( \psi _{\gamma }\right)$,
$$R_{2}\left( z_{n}\right) \sim
\kappa \left( \gamma \right) R_{1}\left( z_{n}\right) ^{2}\sim \kappa
\left( \gamma \right) \left( \gamma +1\right) ^{-2}\left( y_{o}-G^{-1}\left(
1-\ell /n\right) \right) ^{2}
$$ 
\noindent 
and hence 

$$\left( t_{n}\left( \ell /n\right)
-t_{n}\left( \ell \right) \right) ^{2}/R_{2}\left( z_{n}\right) \rightarrow
\left( \gamma +2\right) ^{-1}\left( \lambda ^{-1/\gamma }-1\right) ^{-2},
$$

\noindent all that by Lemmas \ref{l3.4} and \ref{l4.1} which also implies for $0<\varepsilon
<1/\gamma ,$ $\rho _{n}^{-1}\left( 0\right) \geq k^{1/2}\left( k/\ell
\right) $ as $n$ is large. In all cases, $EQ_{n3}^{\ast }\left( 3\right)
\rightarrow 0.$ Finally, $Q_{n3}^{\ast }\left( 3\right) \rightarrow _{p}0$
and hence $Q_{n3}\left( 3\right) \rightarrow _{p}0$.\\

\noindent To finish, with PNOW, we have as $n$ and $\lambda $ are large, 

\begin{equation}
\left| Q_{n3}\left( 2\right) \leq Q_{n3}^{\ast }\left( 2\right) \right| ,
\label{4.11}
\end{equation}
with 
\begin{equation}
E\text{ }Q_{n3}^{\ast }\left( 2\right) \leq 4\ell ^{1/2}\kappa \left(
\gamma \right) ^{-1}\times \frac{t_{n}\left( \lambda k\right) -t_{n}\left(
\ell \right) }{k^{1/4}R_{1}\left( x_{n}\right) }\times \frac{t_{n}\left(
\ell /\lambda \right) -t_{n}\left( \ell \right) }{R_{1}\left( z_{n}\right) }%
\times \frac{R_{1}\left( z_{n}\right) }{k^{1/4}R_{1}\left( x_{n}\right) }.
\label{4.12}
\end{equation}

\noindent One shows exactly as above that 
$$
k^{1/4}R_{1}\left( x_{n}\right) R_{1}\left(z_{n}\right) \rightarrow +\infty 
$$ 

\noindent and 

$$\left( t_{n}\left( \ell /\lambda
\right) -t_{n}\left( \ell \right) \right) /R_{1}\left( z_{n}\right)
$$ 

\noindent is bounded as $n\rightarrow +\infty $ whenever $F\in \Gamma $. Also, 
$$
k^{-1/4}\left( t_{n}\left( \lambda k\right) -t_{n}\left( \ell \right)
\right) /R_{1}\left( x_{n}\right) \rightarrow 0
$$ 

\noindent obviously when $F\in D\left( \psi _{\gamma }\right) $ by Lemmas \ref{l3.4} and \ref{l4.1}.
If $F\in \left( \phi\right) \cup D\left( \Lambda \right) $, $G\in \left( \Lambda \right) $ and,
using (\ref{1.4}) and (\ref{3.15}) and SVZ functions properties, one has any 
$\varepsilon ,$ $0<\varepsilon <1/4$,

\begin{equation*}
k^{-1/4}\left( t_{n}\left( \lambda k\right) -t_{n}\left( \ell \right)
\right) /R_{1}\left( x_{n}\right) 
\end{equation*}

\begin{equation}
\leq \left\{ \left( 1+\varepsilon \right) 
\text{ }s\text{ }\left( \ell /n\right) /s\left( \lambda k/n\right) -1+\left(
1+\varepsilon \right) \times \left( \frac{k}{\ell }\right) ^{\varepsilon
}\log \left( \frac{\lambda k}{\ell }\right) \right\} k^{-1/4}.  \label{4.13}
\end{equation}

\noindent As in the preceding, $k^{-1/4}s\left( \ell /n\right) /s\left( \lambda
k/n\right) \rightarrow 0,k^{-1/4}R_{1}\left( z_{n}\right) /R_{1}\left(
x_{n}\right) \rightarrow 0$. One concludes that $E$ $Q_{n3}^{\ast }\left(
2\right) \rightarrow 0$ and hence $(S1.4)$ holds. 

\begin{equation}
Q_{n1}\left( k\right) \rightarrow \left( k\right) _{p}0.  \label{S2.4}
\end{equation}

\bigskip

\noindent We have

$$
R_{2}\left( x_{n}\right) Q_{n1}\left( k\right) = \left( n/k\right) ^{1//2}\int_{\widetilde{x}%
_{n}}^{x_{n}}\int_{y}^{x_{n}}B_{n}\left( 1-G\left( t\right) \right) \text{ }%
dtdy
$$

$$
+\left( n/k\right) ^{1/2}\int_{\widetilde{x}_{n}}^{x_{n}}\int_{x_{n}}^{%
\widetilde{z}_{n}}B_{n}\left( \&-G\left( t\right) \right) \text{ }%
dtdy 
$$

$$
 =:Q_{n1}\left( k,1\right)R_{2}\left( x_{n}\right)+Q_{n1}\left( k,2\right)R_{2}\left( x_{n}\right).
$$

\bigskip

\noindent It follows that for any $\lambda >1$, one has with PNOW as $n$ is large, 
\begin{equation}
\left| Q_{n1}\left( k,1\right) \right| \leq Q_{n1}^{\ast }\left( k,1\right)
;\ \left| Q_{n1}\left( k,2\right) \right| \leq Q_{n1}^{\ast }\left(
k,2\right)  \label{4.14}
\end{equation}

\noindent with 

\begin{equation*}
\mathbb{E}\text{ }Q_{n1}^{\ast }\left( k,1\right) \leq 3\kappa \left( \gamma
\right) ^{-1}\lambda ^{1/2}\left( t_{n}\left( k/\lambda \right) -t_{n}\left(
\lambda k\right) \right) ^{2}/R_{1}\left( x_{n}\right) ^{2},  \label{4.15}
\end{equation*}

\noindent and 

\begin{equation*}
\mathbb{E}Q_{n1}^{\ast }\left( k,2\right) \leq 3\kappa \left( \gamma \right) ^{-1}%
\frac{t_{n}\left( k/\lambda \right) -t_{n}\left( \lambda k\right) }{%
R_{1}\left( t_{n}\left( k\right) \right) }
\end{equation*}

\begin{equation}\times \frac{R_{1}\left(
x_{n},t_{n}\left( \ell /\lambda ^{\prime }G_{1/2}\right) \right) }{%
R_{1}\left( x_{n},G_{1/2}\right) }\times \frac{R_{1}\left( x_{n},G\right) }{%
R_{1}\left( x_{n},G_{1/2}\right) }
\end{equation}

\noindent where $\lambda ^{\prime }=\lambda $ for $\ell \rightarrow +\infty $ and $%
\lambda ^{\prime }$ is taken large for $\ell $ fixed. Arguments given in (%
\ref{3.24}) show that  $\mathbb{E}Q_{n1}^{\ast }\left( k,1\right) \rightarrow 0$
and a combination of these same arguments, Lemmas 7 and 8 in L\^{o}
(1990), Lemmas \ref{l3.2} and \ref{l3.3} above ensure that $\mathbb{E}Q_{n1}^{\ast }\left(
k,2\right) \rightarrow 0$. We conclude that $Q_{n1}^{\ast }\left( k,1\right)
+Q_{n1}^{\ast }\left( k,2\right) \rightarrow _{p}0$ and thus, by (\ref{4.8}%
), (\ref{S2.4}) holds. 

\begin{equation}
\left( \ell /k\right) ^{1/2}Q_{n1}\left( \ell \right) \rightarrow _{p}0.
\tag{S3.4}
\end{equation}

\noindent One has with PNOW as $n$ is large 

\begin{equation}
\left| \left( \ell /k\right) ^{1/2}Q_{n1}\left( \ell \right) \right| \leq
Q_{n1}\left( \ell ,1\right) +Q_{n1}^{\ast }\left( \ell ,2\right)
\label{4.17}
\end{equation}

\noindent with 

\begin{eqnarray}
E\text{ }Q_{n1}^{\ast }\left( \ell ,1\right) &\leq &3\kappa \left( \gamma
\right) ^{-1}\left( \lambda \ell /k\right) ^{1/2}\left( t_{n}\left(
k/\lambda _{1}\right) -t_{n}\left( k\right) \right)  \label{4.18} \\
&&\left( t_{n}\left( \ell /\lambda _{2}\right) -t_{n}\left( \ell \right)
\right) /R_{1}\left( x_{n}\right) ^{2}
\end{eqnarray}

\noindent and 

\begin{equation}
E\text{ }Q_{n1}^{\ast }\left( \ell ,2\right) \leq 3\kappa \left( \gamma
\right) ^{-1}\left( \ell /k\right) ^{1/2}\left( t_{n}\left( \ell /\lambda
_{2}\right) -t_{n}\left( \ell \right) \right) R_{1}\left( \widetilde{z}%
_{n},G_{1/2}\right) /R_{1}\left( x_{n}\right) ^{2},  \label{4.19}
\end{equation}

\noindent where $\lambda _{1}>1$ and either $\lambda _{2}=\lambda _{1}$ (for $%
\rightarrow +\infty $) or $\lambda _{2}$ is taken large (for $\ell $ fixed).\\

\noindent 
By the arguments many times used above, one has $\mathbb{E}Q_{n1}^{\ast }\left( \ell
,1\right) \rightarrow 0$ and $EQ_{n1}^{\ast }\left( \ell ,2\right)
\rightarrow 0$ and consequently, $\left( \ell /k\right) ^{1/2}Q_{n1}\left(
\ell \right) \rightarrow _{p}0$ whenever 
\begin{equation}
\rho _{n}\left( 1,\delta \right) =\left( \ell /k\right) ^{\delta
}R_{1}\left( z_{n}\right) /R_{1}\left( x_{n}\right) \rightarrow 0,
\label{4.20}
\end{equation}

\noindent for all $\delta >0$ and for all $df$ $F\in \Gamma $. But, $G\in D\left(
\Lambda \right) \cup D\left( \psi \right) $ and by Lemma \ref{l3.4}, 
\begin{equation}
\rho _{n}\left( 1,\delta \right) \leq 2\left( \ell /k\right) ^{\lambda
+\delta -\varepsilon },  \label{4.21}
\end{equation}

\noindent for $0<\varepsilon <\delta$, $\lambda =1/\gamma$  for $F\in D\left( \psi
_{\gamma }\right) $ or $\lambda =0$, $F\in D\left( \Lambda \right) \cup D\left(
\phi \right)$. This completes the proof of $(S3.4)$.

\begin{equation}
\ell k^{-1/2}Q_{n2}\left( \ell \right) \rightarrow_{P}0. \tag{S4.4}
\end{equation}

\noindent One has with PNOW $\left(\ell\right)\leq n$ is large, 

\begin{equation}
\left| \ell k^{-1/2}Q_{n2}\left( \ell \right) \right| \leq Q_{n2}^{\ast
}\left( \ell ,1\right) +Q_{n2}^{\ast }\left( \ell ,2\right) \label{4.22}
\end{equation}

\noindent with 

\begin{equation*}
Q_{n2}^{\ast }\left( \ell ,1\right) \leq \lambda \ell k^{-1/2}\left(
t_{n}\left( \ell \right) -t_{n}\left( \lambda _{1}k\right) \right) \text{ }%
t_{n}\left( \ell /\lambda _{2}\right) -t_{n}\left( \ell \right) /R_{2}\left(
t_{n}\left( k\right) \right)  \label{4.23}
\end{equation*}

\noindent and 

\begin{equation}
Q_{n2}^{\ast }\left( \ell ,2\right) \leq \ell k^{-1/2}\left( t_{n}\left(
\ell /\lambda _{2}\right) -t_{n}\left( \ell \right) \right) R_{1}\left( 
\widetilde{z_{n}}\right) /R_{2}\left( t_{n}\left( k\right) \right) ,
\label{4.24}
\end{equation}

\noindent where $\lambda _{1}>0$, $\lambda _{2}=\lambda _{1}$ or $\lambda _{2}$ is
taken large. Always by the now familiar arguments used above and by general properties of
SVZ functions, one shows that 
\begin{equation}
Q_{n2}^{\ast }\left( \ell ,2\right) \rightarrow _{p}0,  \label{4.25}
\end{equation}

\noindent and 

\begin{equation}
Q_{n2}^{\ast }\left( \ell ,1\right) \rightarrow _{p}0,  \label{4.26}
\end{equation}

\bigskip

\noindent (i) for $F\in D\left( \psi _{\gamma }\right) $ whenever $\ \lim
\sup_{n\rightarrow +\infty }$ or there exists $\xi $, $0<\xi <1/2$, such
that $\ell ^{1+1/\gamma }=o\left( k^{1/2-\xi +1/\gamma }\right)$;\\

\noindent (ii) for $F\in D\left( \phi _{\alpha }\right) $ whenever $\ell k^{-1/2}\log
k\rightarrow 0$ since $R_{1}\left( t\right) \sim R_{2}\left( t\right)
^{1/2}\rightarrow \alpha ^{-1}$ as $t\rightarrow +\infty $ and $t_{n}\left(
\ell \right) -t_{n}\left( \lambda _{1}k\right) =0\left( \log k\right) $ (see
(3.7a) in L\^{o} (1990) where $k\sim n^{\vartheta },$ $0<\vartheta <1)$;\\

\noindent (iii) and for $F\in D\left( \Lambda \right) $ whenever there exists $\eta ,$ 
$0<\eta <1/2$, such that $\ell =0\left( k^{1/2-\eta }\left( \log k\right)
^{2\eta +1}\right) $ since for any $\varepsilon <0$, as n is large, 

\begin{equation}
0\leq \left\{ t_{n}\left( \ell \right) -t_{n}\left( \lambda _{1}k\right)
\right\} /R_{1}\left( x_{n}\right) \leq 2\left\{ s\left( \ell
/n\right) \text{ }/\text{ }s\text{ }\left( k/n\right) \right\} +1+\left(
k/\ell \right) ^{\varepsilon }\log \left( \lambda _{1}k/\ell \right) .
\label{4.27}
\end{equation}

\noindent All these conditions are implied by the hypotheses on $\left( k,\ell \right)$. This completes the proof of $(S4.4)$. It remains to prove this important result.

\bigskip

\begin{lemma} \label{l4.3} Let $F\in \Gamma $, $k\rightarrow +\infty ,$ $k/n\rightarrow 0,$ $\ell
/k\rightarrow 0,$ then $N_{n}\left( 3,k,\ell \right) $ is a Gaussian $rv$
with mean zero such that $\mathbb{E}N_{n}\left( 3,k,\ell \right) ^{2}\sim
6R_{4}\left( x_{n},z_{n}\right) /R_{2}\left( x_{n}\right) ^{2}\sim \sigma
_{2}\left( \gamma \right) ,0<\gamma <+\infty $.
\end{lemma}

\noindent \textbf{Proof of Lemma \ref{l4.3}}. That $N_{n}\left( 3,k,\ell \right) $ is Gaussian follows
from Theorem C. It's variance is 

\begin{equation*}
s_{n}^{2}\left( 3\right) =\left( \frac{n}{k}\int_{x_{n}}^{z_{n}}%
\int_{x_{n}}^{z_{n}}\int_{p}^{z_{n}}\int_{q}^{z_{n}}h\left( s,t\right) \text{
}dt\text{ }ds\text{ }dp\text{ }dq\right) /R_{2}\left( x_{n}\right) ^{2}
\end{equation*}

\begin{equation}
=:\left( \frac{n}{k}\int_{x_{n}}^{z_{n}}\int_{x_{n}}^{z_{n}}H\left(
p,q\right) \text{ }dp\text{ }dq\right) /R_{2}\left( x_{n}\right) ^{2}, \label{4.28} \\
\end{equation}

\noindent where $h\left( s,t\right) =\min \left( 1-G\left( t\right) ,\text{ }1-G\left(
s\right) \right) -\left( 1-G\left( t\right) \right) \left( 1-G\left(
s\right) \right) .$

\noindent Using the symmetry of $H\left( \circ ,\circ \right) $ and considering the
case $p\leq t$ and $p>t$ yield 

\begin{equation}
s_{n}^{2}\left( 3\right) \sim \left\{ \left( 2n/k\right)
\int_{x_{n}}^{z_{n}}\int_{p}^{z_{n}}H\left( p,q\right) \text{ }dp\text{ }%
dq\right\} /R_{2}\left( x_{n}\right) ^{2}.  \label{4.29}
\end{equation}

\noindent Further, cutting the integration space into $\left\{ s\leq t\right\} $ and $%
\left\{ s>t\right\} $ gives 
\begin{equation}
s_{n}^{2}\left( 3\right) \sim 2\left( 2R_{4}\left( x_{n},z_{n}\right) \left(
1+o\left( 1\right) \right) +\left( \frac{n}{k}\int_{x_{n}}^{z_{n}}%
\int_{p}^{z_{n}}\left( q-p\right) \int_{q}^{z_{n}}1-G\left( t\right) \text{ }%
dt\text{ }dq\text{ }dp\right) \right)  \label{4.31}
\end{equation}
The second term in brackets is $\sim R_{4}\left( x_{n},z_{n}\right) $ by an
integration by parts with

$v=q-p;$ $u=\int_{q}^{z_{n}}\int_{y}^{z_{n}}1-G\left( t\right) $ $dtdy$.
Finally, 
\begin{equation}
s_{n}^{2}\left( 3\right) \sim 6R_{4}\left( x_{n},z_{n}\right) /R_{2}\left(
x_{n}\right) ^{2}.  \label{4.3}
\end{equation}

\noindent Lemmas \ref{l4.1} and \ref{l4.2} thus complete the proof of Lemma \ref{l4.3}. Theorem \ref{t4.1}. is proved by (S1.4), ..., (S4.4) and Lemma \ref{l4.3}.

\bigskip

\noindent The first part of Remark \ref{r4.2} follows from the lines just below (\ref{4.26}). The second part of Remark \ref{r4.2} follows by remarking that 

\begin{equation*}
k^{1/2}\left| \left\{ \tau \left( \widetilde{k},\ell \right) -\tau \left( 
\widetilde{k}\right) \right\} /R_{2}\left( x_{n}\right) \right| 
\end{equation*}

\begin{equation*}
\leq \ell k^{-1/2}\left( t_{n}\left( \ell \right) -t_{n}\left( \lambda k\right)
\right) R_{1}\left( t_{n}\left( \ell \right) /R_{2}\left( t_{n}\left(
k\right) \right) \right) 
\end{equation*}

\begin{equation}
 +\ell k^{-1/2}R_{2}\left( t_{n}\left( k\right)
\right) ,  \label{4.32}
\end{equation}

\noindent for $\lambda >1$, with PNOW as n is large. Both terms at right tend to zero
exactly as in (\ref{4.24}) and (\ref{4.25}). Remark \ref{r4.2} is now completely
justified.

\noindent To prove Theorem \ref{t4.2}, remark that, by Theorem \ref{t4.1}, 
\small
\begin{equation*}
A_{n}^{\ast } =N_{n}\left( 3,k,\ell \right) +k^{1/2}\left( \frac{n}{k}%
\int_{\widetilde{x}_{n}}^{z_{n}}\int_{y}^{z_{n}}1-G\left( t\right) \text{ }dt%
\text{ }dy-\frac{n}{k}\int_{x_{n}}^{z_{n}}\int_{y}^{z_{n}}1-G\left( t\right) 
\text{ }dtdy\right) /R_{2}\left( x_{n}\right) +o_{p}\left( 1\right)
\end{equation*}

\Large
\begin{equation}
=:N_{n}\left( 3,k,\ell \right) +Q_{n2}\left( k\right) +o_{p}\left(
1\right) \label{4.33}
\end{equation}

\noindent But, as in (\ref{3.40}), 

\small
\begin{equation*}
Q_{n2}\left( k\right) =k^{1/2}\left( \frac{n}{k}\int_{\widetilde{x}%
_{n}}^{x_{n}}\int_{y}^{x_{n}}1-G\left( t\right) \text{ }dt+k^{1/2}\left( 
\frac{n}{k}\int_{\widetilde{x}_{n}}^{z_{n}}\int_{x_{n}}^{z_{n}}1-G\left(
t\right) \text{ }dtdy\right) \right) \\
/R_{2}\left( x_{n}\right)
\end{equation*}

\Large
\begin{equation}
=\left( 2k^{1/2}\kappa \left( \gamma \right)
\right) ^{-1\xi }\left( k\right) ^{2}\left( 1\right) +\kappa \left(
\gamma \right) ^{-1\xi }\left( k\right) \left( 1+o_{p}\left( 1\right)
\right) ,
\end{equation}

\noindent where $\xi _{n}\left( \gamma \right) $ is defined in Lemma \ref{l3.6} by which 

\begin{equation*}
A_{n}^{\ast}/(2k^{1/2}\kappa \left( \gamma \right) ^{-1})=
e_{1}\left( \gamma \right) N_{n}\left( 2,k\right) +e_{2}\left( \gamma
\right) N_{n}\left( -\infty \right) \left( 1+o_{p}\left( 1\right) \right)
^{2}
\end{equation*}

\begin{equation}
+e_{4}\left( \gamma \right) N_{n}\left( -\infty \right) .
\label{4.35}
\end{equation}

\noindent It remains to compute $EN_{n}\left( 4,k,\ell \right) ^{2}$. Remark that $%
h\left( t,\frac{k}{n}\right) =\left( 1-G\left( t\right) \right) \left(
1-k/n\right) $ for $x_{n}\leq t\leq z_{n}$ (see (\ref{4.28})) and thus 

\begin{equation*}
\mathbb{E}N_{n}\left( 3,k,\ell \right) N_{n}\left( 2,k\right) 
\end{equation*}

\begin{equation}
-\left( n/k\right)
\int_{x_{n}}^{z_{n}}\int_{y}^{z_{n}}h\left( t,\frac{k}{n}\right) \text{ }dt%
\text{ }dy/R_{2}\left( x_{n}\right) \rightarrow -1.  \label{4.36}
\end{equation}

\noindent This easily implies that 
\begin{equation}
\mathbb{E}N_{n}\left( 4,k,\ell \right) ^{2}=\sigma _{3}^{2}\left( \gamma \right)
-2e_{3}\left( \gamma \right) ^{2},\sigma _{3}^{2}\left( \infty \right) =5.
\label{4.37}
\end{equation}

\noindent We have proved the first part of Theorem \ref{t4.1}. The characterization is
obvious now when we remember that $k^{-1/2}N_{n}\left( -\infty \right)
=f\left( U_{k+1,n}\right) -f\left( k/n\right) \rightarrow _{p}0$.

\bigskip

\section{Limit laws for $\widetilde{C}_{n}=Y_{n-\ell ,n}-Y_{n-k,n}=%
\widetilde{z}_{n}-\widetilde{x}_{n}$} \label{sec5}

\noindent We assume from now on that the regularity conditions (\ref{1.7}) or (\ref
{1.8}) or (\ref{1.9}) hold for sake of simplicity. But it
will appear in the proofs how optimum results may be obtained. Notice
that (\ref{1.8}) or (\ref{1.9}) is required only when $\ell
\rightarrow +\infty $. Put $C_{n}=z_{n}-x_{n}$.

\bigskip

\begin{theorem} \label{t5.1}. Let $F\equiv \left( f,b\right) \in \Gamma $ satisfying the regularity
conditions. Let (\ref{2.4a}) holds.\\

\noindent 1) If $F\in D\left( \Lambda \right) \cup D\left( \phi \right) $, then

$$
\left( \widetilde{C}_{n}-C_{n}\right) /R_{1}\left( z_{n}\right) =\left( 
\widetilde{z}_{n}-z_{n}\right) /R_{1}\left( z_{n}\right) +o_{p}\left(
1\right) 
$$ 

$$
=-\log E_{n}\left( \ell \right) +o_{p}\left( 1\right) \rightarrow
^{d}-\log E\left( \ell \right),
$$ 

\noindent when is $\ell $ fixed and\\

$$
\ell ^{1/2}\left( \widetilde{C}_{n}-C_{n}\right) /R_{1}\left( z_{n}\right)
=\ell ^{1/2}\left( \widetilde{z}_{n}-z_{n}\right) /R_{1}\left( z_{n}\right)
+o_{p}\left( 1\right)$$

$$
=e_{1}\left( \gamma \right) N_{n}\left( 2,\ell \right)
+o_{p}\left( 1\right),
$$ 

\noindent when $\ell \rightarrow +\infty $ and $\ell /k\rightarrow 0$.\\

\noindent 2) Let $F\in D\left( \psi _{\gamma }\right)$\\

\noindent a) If $\gamma >2$, then

$$
\left( \widetilde{C}_{n}-C_{n}\right) /R_{1}\left( z_{n}\right) =\left( 
\widetilde{z}_{n}-z_{n}\right) /R_{1}\left( z_{n}\right) +o_{p}\left(
1\right)
$$

$$
=\left( \gamma +1\right) \left( 1-E_{n}\left( \ell \right)
^{1/\gamma }\right) +o_{p}\left( 1\right) \rightarrow ^{d} \left( \gamma +1\right) \left( 1-E\left( \ell \right) ^{1/\gamma }\right)
$$

\noindent when $\ell$ is fixed and\\

$$
\ell ^{1/2}\left( \widetilde{C}_{n}-C_{n}\right) /R_{1}\left( z_{n}\right)
=\ell ^{1/2}\left( \widetilde{z}_{n}-z_{n}\right) /R_{1}\left( z_{n}\right)
+o_{p}\left( 1\right)
$$

$$
=e_{1}(\gamma) N_{n}(2,\ell)+o_{p}(1),
$$

\noindent when $\ell \rightarrow +\infty$, $\ell/k\rightarrow 0$.\\

\noindent b) If $0<\gamma <2$, then

$$
k^{1/2}\left( \widetilde{C}_{n}-C_{n}\right) /R_{1}\left( x_{n}\right)
=e_{1}\left( \gamma \right) N_{n}\left( 2,k\right) +o_{p}\left( 1\right)
\rightarrow ^{d}N\left( 0,e_{1}\left( \gamma \right) ^{2}\right),
$$

\noindent in both cases where $\ell $ is fixed and $\ell \rightarrow +\infty $ while $%
\ell /k\rightarrow 0$.
\end{theorem}

\begin{remark} \label{r5.1} These results notably extend earlier results by de Haan and Resnick
(1980) and by L\^{o} (1986b).
\end{remark}

\begin{remark} \label{r5.2} $\widetilde{z}_{n}$ dominates $\widetilde{x}_{n}$ is these results
since $\widetilde{C}_{n}$ follows the law provided by $\widetilde{z}_{n}$
except when $F\in D\left( \psi _{\gamma }\right) ,$ $0<\gamma <2.$ This fact
will occur many times in the sequel.
\end{remark}

\begin{corollary} \label{c5.1} Let $F\equiv \left( f,b\right) \in \Gamma $ $\ $\ satisfying the
regularity conditions and let (\ref{2.4a}) hold.\\

\noindent a) Put $c_{n}\left( 8\right) =z_{n}-x_{n},$ $c_{n}\left( 9\right) =\left(
y_{o}-x_{n}\right) ,$ $v_{n}\left( 8\right) =\ell ^{1/2}/R_{1}\left(
z_{n}\right) $ and $v_{n}\left( 9\right) =\ell ^{1/2}R_{1}\left(
x_{n}\right) $\\

\noindent If $F\in D\left( \Lambda \right) \cup D\left( \phi \right)$, then

$$
\left( n^{\upsilon }T_{n}\left( 8\right) ^{-1}-c_{n}\left( 8\right) \right)
/R_{1}\left( z_{n}\right) \rightarrow ^{d}-\log E\left( \ell \right)
$$ 

\noindent when $\ell$ is fixed and\\

$$
v_{n}\left( 8\right) \left( n^{\upsilon }T_{n}\left( 8\right)
^{-1}-c_{n}\left( 8\right) \right) =-N_{n}\left( 2,\ell \right) +o_{p}\left(
1\right) \rightarrow ^{d}N_{n}\left( 0,1\right)
$$

\noindent when $\ell \rightarrow+\infty$, $\ell /k\rightarrow 0.$\\

\noindent If $F\in D\left( \psi _{\gamma }\right)$, $\gamma >2$, then for $v_{n}^{\ast}\left( 8\right) =k^{1/2}/R_{1}\left( x_{n}\right)$,\\

$$
v_{n}^{\ast }\left( 8\right) \left( n^{\upsilon }T_{n}\left( 8\right)
^{-1}-c_{n}\left( 8\right) \right) =e_{1}\left( \gamma \right) N_{n}\left(
2,k\right) +o_{p}\left( 1\right),
$$

\noindent  when $\ell $ is fixed of $\ell \rightarrow +\infty $ while $\ell
/k\rightarrow 0$.\\

\bigskip

\noindent If $F\in D\left( \psi _{\gamma }\right) ,$ $0<\gamma <+\infty ,$ then

\noindent $R_{1}\left( x_{n}\right) \left( T_{n}\left( 9\right) -c_{n}\left( 9\right)
\right) /R_{1}\left( z_{n}\right) \rightarrow ^{d}\left( E\left( \ell
\right) ^{1/\gamma }-1\right) $ when $\ell $ is fixed and\\

\bigskip 

\noindent $v_{n}\left( 9\right) \left( T_{n}\left( 9\right) \right) -c_{n}\left(
9\right) =\gamma ^{-1}N_{n}\left( 2,\ell \right) +o_{p}\left( 1\right) $
when $\ell \rightarrow +\infty $ while $\ell /k\rightarrow 0$.
\end{corollary}

\bigskip

\begin{remark} \label{r5.3} For $\gamma =2,$ the results depend crucially on b$\left(.\right) $
(see Lemma \ref{l6.1} and remark \ref{r6.1} below). Mixture cases are signaled in these
examples.
\end{remark}

\bigskip

\begin{proof} First, the following lemma is a direct consequence of properties of $\rho
-RVZ$ functions. $\rho \geq 0$.

\begin{lemma} \label{l5.1}. Let 

$$
\rho_{n}\left( 2,1/2\right) =k^{-1/2}R_{1}\left( x_{n}\right)
\left( z_{n}\right), \text{ } \rho _{n}\left( 3,1/2\right)
$$

$$
=\ell^{1/2}k^{-1/2}R_{1}\left( x_{n}\right) /R_{1}\left( z_{n}\right)
$$.

\bigskip

If $F\in \Gamma ,$ then $\rho _{n}\left( 2,1/2\right) \rightarrow 0$ as $%
n\rightarrow +\infty ,$ $k\rightarrow +\infty ,$ $k/n\rightarrow 0,$ $\ell $
being fixed and\\

$\rho _{n}\left( 3,1/2\right) \rightarrow 0$ as $n\rightarrow +\infty ,$ $%
\ell \rightarrow +\infty ,$ $k\rightarrow +\infty ,$ $k/n\rightarrow 0$, $%
\ell /k\rightarrow 0,$ for $2<\gamma \leq +\infty $.

\noindent For $0<\gamma <2,$ both limits are infinite.
\end{lemma}

\bigskip

\noindent a) Now, let $\ell $ $\rightarrow +\infty $. Using Lemma \ref{l3.6} and the fact
that $m^{1/2}\gamma _{n}\left( m\right) =o_{p}\left( 1\right) $ for $m=k$ or 
$\ m=\ell $, we get after routine considerations, 

\begin{equation}
\ell ^{1/2}\left( \widetilde{C}_{n}-C_{n}\right) /R_{1}\left( z_{n}\right)
e_{1}\left( \gamma \right) N_{n}\left( 2,\ell \right) +\rho _{n}\left(
3,1/2\right) e_{1}\left( \gamma \right) N_{n}\left( 2,k\right) +o_{p}\left(
1\right)  \label{5.1}
\end{equation}

\noindent which combined with Lemma \ref{l5.1} yields 

\begin{equation}
\ell ^{1/2}\left( \widetilde{C}_{n}-C_{n}\right) /R_{1}\left( z_{n}\right)
=e_{1}\left( \gamma \right) N_{n}\left( 2,\ell \right) +o_{p}\left( 1\right)
\label{5.2a}
\end{equation}

\noindent for $F\in D\left( \Lambda \right) \cup D\left( \psi _{\gamma }\right) $, $%
2<\gamma \leq +\infty ,$ and 

\begin{equation}
v_{n}^{\ast }\left( 8\right) \left( \widetilde{C}_{n}-C_{n}\right)
=e_{1}\left( \gamma \right) N_{n}\left( 2,\ell \right) +o_{p}\left( 1\right)
\label{5.2b}
\end{equation}

\noindent for $F\in D\left( \psi _{\gamma }\right) ,$ $o<\gamma <2$. This completes
the proofs of Theorems \ref{t5.1}.\\

\noindent b) Let $\ell $ be fixed.\\

\noindent For $F\in D\left( \psi _{\gamma }\right) ,$ one has by part (i) of Lemma \ref{l4.1}
and Formula (\ref{2.6}) 

\begin{equation*}
\left( \widetilde{z}_{n}-z_{n}\right) /R_{1}\left( z_{n}\right) \sim \left(
\gamma +1\right) \left( \frac{y_{o}-\widetilde{z}_{n}}{y_{o}-z}\right)
\end{equation*}

\begin{equation}
\sim
\left( \gamma +1\right) \left\{ \left( nU_{\ell +1,n}/\ell \right) ^{\pm
\varepsilon _{n}+1/\gamma }-1\right\} \left( 1+o_{p}\left( 1\right) \right) ,
\label{5.3}
\end{equation}

\noindent where $\varepsilon _{n}=\sup_{0\leq u\leq U_{\ell +1,n}}\left| b\left(
t\right) \right| \rightarrow _{p}0$.\\

\noindent This and Theorem C together prove Part a) i) of Theorem \ref{t5.1}.\\

\noindent For $F\in D\left( \phi \right) \cup D\left( \Lambda \right) ,$ $G\in D\left(
\Lambda \right) $. Using (\ref{1.4}) gives 

\begin{equation}
z_{n}-\widetilde{z}_{n}=s\left( U_{\ell +1,n}\right) -s\left( \ell /n\right)
+\int \left( \ell /n,U_{\ell +1,n}\right) s\left( t\right) t^{-1}dt,
\label{5.4}
\end{equation}

\noindent Set $\Theta _{n}=\left\{ s\left( t\right) /s\left( \ell /n\right) ,\ell
/n\leq t\leq U_{\ell +1,n}\text{ or }U_{\ell +1,n}\leq t\leq \ell /n\right\} 
$. It is easily shown from (\ref{1.3a}) that $\ \inf \Theta _{n}\rightarrow
_{p}1$ and $\ \sup \Theta _{n}\rightarrow _{p}1$. It follows that 

\begin{equation}
\noindent z_{n}-\widetilde{z}_{n}=\log E_{n}\left( \ell \right) +o_{p}\left( 1\right) .
\label{5.5}
\end{equation}

\noindent We do remark that $-\log E\left( \ell \right) $ is the Gumbel extremal law
for $\left( \ell +1\right) ^{th}$ maximum. Finally, 

\begin{equation}
\left( \widetilde{C}_{n}-C_{n}\right) /R_{1}\left( z_{n}\right) =-\log
E_{n}\left( \ell \right) +\rho _{n}\left( 2,k\right) +o_{p}\left( 1\right)
=-\log E_{n}\left( \ell \right) +o_{p}\left( 1\right) \text{.}  \label{5.6}
\end{equation}

\noindent This proves part a) ii) of Theorem \ref{t5.1}.\\

\noindent The results related to $T_{n}\left( 8\right) $ in Corollary \ref{c5.1} are
immediate. To prove those related to $T_{n}\left( 9\right) ,$ check that 
\begin{equation}
T_{n}\left( 9\right) -c_{n}\left( 9\right) =\left( \left( y_{o}-x_{n}\right)
\left( y_{o}-\widetilde{x}_{n}\right) \right) ^{-1}\left( \left(
y_{o}-x_{n}\right) \left( z_{n}-\widetilde{z}_{n}\right) -\left(
y_{o}-z_{n}\right) \left( x_{n}-\widetilde{x}_{n}\right) \right) ;
\label{5.7}
\end{equation}

\noindent and hence by Lemmas \ref{l3.2} and \ref{l4.1}, 
\begin{equation}
R_{1}\left( x_{n}\right) \left( T_{n}\left( 9\right) -c_{n}\left( 9\right)
\right) R_{1}\left( z_{n}\right) \sim \left( z_{n}\right) \sim \left( \gamma
+1\right) ^{-1}\left( \frac{z_{n}-\widetilde{z}_{n}}{R_{1}\left(
z_{n}\right) }-e_{1}\left( \gamma \right) N_{n}\left( 2,k\right)
k^{-1/2}\right) .  \label{5.8}
\end{equation}

\noindent And this proves the results related to $\ T_{n}\left( 9\right) $ in
Corrollary 5.1 following the lines just above.
\end{proof}

\noindent In the next section, we deal with all the ratios of statistics already
including the major element of the ECSFEXT which is $T_{n}( 1,k,\ell)$.

\newpage

\section{Limit law for ratios from the ECSFEXT} \label{sec6}

\subsection{Case of $T_{n}\left( 1,k,\ell \right) $.}

First, we obtain a general result

\begin{theorem} \label{t6.1}. Let $F\equiv \left( f,b\right) \in \Gamma $. If $\left( k,\ell \right) $
satisfies (\ref{2.4a}) or (\ref{2.4b}), then

$$
k^{1/2}\left( T_{n}\left( 1,k,\ell \right) -\mu \left( \widetilde{k}\right) 
/\tau \left( \widetilde{k}\right) ^{1/2}\right)
$$

$$
=\left(2\sqrt{\kappa \left( \gamma \right) }\right) ^{-1}\left( 2N_{n}\left(
0,k,\ell \right) -N_{n}\left( 3,k,\ell \right) \right) +o_{p}\left( 1\right)
$$

$$
 \rightarrow ^{d}N\left( 0,\sigma _{4}^{2}\left( \gamma \right) \right)
,0<\gamma \leq +\infty
$$

\noindent where $EN_{n}\left( 0,k,\ell \right) N_{n}\left( 3,k,\ell \right) \sim
3\left( \gamma +1\right) /\left( \gamma +3\right)$ so that 
$$
\sigma _{4}^{2}\left( \gamma \right) =\frac{2\gamma ^{3}+10\gamma ^{2}+32\gamma +24%
}{4\left( \gamma +1\right) \left( \gamma +3\right) \left( \gamma +4\right) }%
, 0<\gamma \leq +\infty.
$$
\end{theorem}

\bigskip 

\begin{theorem} \label{t6.2}. Let $F\equiv \left( f,b\right) \in \Gamma $ and $\left( k,\ell \right) $
satisfies (\ref{2.4a}) or (\ref{2.4b}). Under the regularity conditions, we have

$v_{n}\left( 1\right) \left( T_{n}\left( 1,k,\ell \right) -c_{n}\left(
1\right) \right) =\left( 2\sqrt{\kappa \left( \gamma \right) }\right)
^{-1}\left( 2N_{n}\left( 1,k,\ell \right) -N_{n}\left( 4,k,\ell \right)
\right) +o_{p}\left( 1\right) \rightarrow ^{d}N\left( 0,\sigma _{5}^{2}\left( \gamma \right) \right) ,$ $%
0<\gamma \leq +\infty $.\\

\bigskip

\noindent where $v\left( 1\right) =k^{1/2},$ $c\left( 1\right) =\mu \left( k\right) /%
\sqrt{\tau \left( k\right) }$ and 
$$
\sigma ^{2}\left( \gamma \right) =\frac{%
\gamma ^{3}+\gamma ^{2}+2\gamma }{4\left( \gamma +1\right) \left( \gamma
+3\right) \left( \gamma +4\right) }, \text{ } 0<\gamma \leq +\infty.
$$

\end{theorem}

\begin{proof}

We need only to prove Theorem \ref{t6.1} as Corollary of Theorems \ref{t3.1} and \ref{t4.1}
Theorems \ref{t6.2} follows from Theorems \ref{t3.2} and \ref{t4.2} by the very same arguments.
By Theorem \ref{t3.1}, 
\begin{equation}
T_{n}\left( 2,k,\ell \right) =\mu \left( \widetilde{k}\right) +\mu \left(
k\right) k^{1/2}N_{n}\left( 0,k,\ell \right) +o_{p}\left( k^{-1/2}\mu \left(
k\right) \right),  \label{6.1}
\end{equation}
so that 
\begin{equation}
T_{n}\left( 2,k,\ell \right) ^{2}=\mu \left( \widetilde{k}\right)
^{2}+2k^{-1/2}N_{n}\left( 0,k,\ell \right) +o_{p}\left( k^{-1/2}\mu \left(
k\right) ^{2}\right).  \label{6.3}
\end{equation}

\noindent Since $\mu \left( \widetilde{k}\right) \sim \mu \left( k\right) $ in
probability, one has 
\begin{equation}
T_{n}\left( 2,k,\ell \right) =\mu \left( \widetilde{k}\right) \left(
1+o_{p}\left( 1\right) \right).
\end{equation}

\noindent Furthermore, by Theorem \ref{t4.1}, 
\begin{equation}
A_{n}\left( 1,k,\ell \right) =\tau \left( \widetilde{k}\right) +\tau \left(
k\right) k^{-1/2}N_{n}\left( 3,k,\ell \right) +o_{p}\left( k^{-1/2}\tau
\left( k\right) \right)  \label{6.4}
\end{equation}
so that 
\begin{equation}
A_{n}\left( 1,k,\ell \right) =\tau \left( \widetilde{k}\right) \left(
1+o_{p}\left( 1\right) \right) .  \label{6.5}
\end{equation}

\noindent Now a straightforward calculation based on (\ref{6.1})-(\ref{6.5}) and on
the fact that $\tau \left( k\right) \sim \kappa \left( \gamma \right) \mu
\left( k,\ell \right) ^{2}$ (see Lemma \ref{l3.2} and \ref{l4.1}) yields 

\begin{equation*}
k^{1/2}\left( T_{n}\left( 1,k,\ell \right) -\mu \left( \widetilde{k}\right)
/\tau \left( \widetilde{k}\right) ^{1/2}\right) =k^{1/2}\left( T_{n}\left(
1,k,\ell \right) ^{2}-\mu \left( \widetilde{k}\right) ^{2}/\tau \left( 
\widetilde{k}\right) \right)  
\end{equation*}

\begin{equation}
k^{1/2}\left( T_{n}\left( 1,k,\ell \right) -\mu \left( \widetilde{k}\right)
/\tau \left( \widetilde{k}\right) ^{1/2}\right) = \left( 2\sqrt{%
\kappa \left( \gamma \right) }\right) ^{-1}\left( 2N_{n}\left( 0,k,\ell
\right) -N_{n}\left( 3,k,\ell \right) \right) +o_{p}\left( 1\right). \label{6.6}
\end{equation}

\noindent This partly proves Theorem \ref{t6.1}. It remains to compute $\mathbb{E}N_{n}\left( 0,k,\ell \right) N_{n}\left( 3,k,\ell \right) =s_{n}\left(0,3\right) $ which is 

\begin{equation}
s_{n}\left( 0,3\right) =\left( \left( n/k\right)
\int_{x_{n}}^{z_{n}}ds\int_{x_{n}}^{z_{n}}dy\int_{y}^{z_{n}}h\left(
s,t\right) \text{ }dt\right) /R_{1}\left( x_{n}\right) R_{2}\left(
x_{n}\right)  \label{6.7}
\end{equation}

\begin{equation*}
= \left( n/k\right) \left( \int_{x_{n}}^{z_{n}}dy\int_{y}^{z_{n}}dt\int_{x_{n}}^{z_{n}}h \left(s,t\right) \text{ } ds\right) /\left( R_{1}\left( x_{n}\right) R_{2}\left(x_{n}\right) \right),
\end{equation*}

\noindent where $h\left( s,t\right) $ is defined in (\ref{4.28}). Cutting the space
integration into $s\leq y$ and $s>y$ and using the first (resp. the second)
expression of $s_{n}\left( 0,3\right) $ for $s\leq y$ (resp. $s>y$), we obtain 

\begin{equation}
s_{n}\left( 0,3\right) \sim \left\{ \left( \left( n/k\right)
\int_{x_{n}}^{z_{n}}\int_{s}^{z_{n}}\int_{y}^{z_{n}}1-G\left( t\right) \text{
}dt\text{ }dy\text{ }ds\right) /R_{1}\left( x_{n}\right) R_{2}\left(
x_{n}\right) \right\}.  \label{6.8}
\end{equation}

\noindent Considering now the case $s\leq t$ and $s>t$ in the second term of (\ref{6.8})
gives 
\begin{equation}
s_{n}\left( 0,3\right) =\left( R_{3}\left( x_{n}\right) \left( 1+o_{p}\left(
1\right) \right) +2R_{3}\left( x_{n}\right) \left( 1+o_{p}\left( 1\right)
\right) /R_{1}\left( x_{n}\right) R_{2}\left( x_{n}\right) \right) ,
\label{6.9}
\end{equation}

\noindent where we used Lemma \ref{l3.2}. Now Lemmas \ref{l4.1} and \ref{l4.2} imply 

\begin{equation}
s_{n}\left( 0,3\right) \sim 3\left( \gamma +1\right) \left( \gamma +3\right)
^{-3},\text{ }0<\gamma \leq +\infty .  \label{6.10}
\end{equation}

\noindent from which we derive $\sigma _{4}\left( \gamma \right) $. Theorem \ref{t6.1} is now
entirely proved. Theorem \ref{t6.2} is proved by the same arguments but we must say
a few words on $\sigma _{5}\left( \gamma \right) $. One has 
\begin{equation}
2N_{n}\left( 1,k,\ell \right) -N_{n}\left( 4,k,\ell \right) =2N_{n}\left(
0,k,\ell \right) -N_{n}\left( 3,k,\ell \right) +N_{n}\left( 2,k\right) ,
\label{6.11}
\end{equation}

\noindent which, together with (\ref{3.43}) and (\ref{6.10}), permits to compute $%
\sigma _{5}\left( \gamma \right) $.
\end{proof}

\bigskip

\subsection{Case of $T_{n}\left( 3,k,\ell \right)$}.

We already noticed in Theorem \ref{t5.1} that when $\widetilde{z}_{n}-\widetilde{x}%
_{n}$ intervenes, the contribution of $\widetilde{z}_{n}$ (normal or
extremal) dominates that of $\widetilde{x}_{n}$ (normal).\\

\noindent Here again, this is the case except when $F\in D\left( \psi _{\gamma
}\right) $ where each of $\widetilde{z}_{n}$ and $\widetilde{x}_{n}$ may
get the better of the other with possibilities of a mixture of both. We
beginn with.

\bigskip

\begin{lemma} \label{l6.1}. 
Let $F\in \Gamma $. Put $\rho _{n}\left( 4\right) =k^{-1/2}\left(
z_{n}-x_{n}\right) R_{1}\left( z_{n},G\right) ^{-1},$ $\rho _{n}\left(
5\right) =\ell ^{1/2}\rho _{n}\left( 4\right)$. The results below hold for $\ell$ concerning $\rho _{n}(4)$ and 
$\ell \rightarrow +\infty$, and $\ell /k\rightarrow0$ concerning $\rho _{n}(5)$.

\noindent 1) If $F\in D\left( \Lambda \right) \cup D\left( \phi \right) $, then $%
\left( \rho _{n}\left( 4\right) ,\text{ }\rho _{n}\left( 5\right) \right)
\rightarrow \left( 0,0\right).$\\

\noindent 2) Let $F\in D\left( \psi _{\gamma }\right) ,$ $\gamma >0.$\\

\noindent a) If $\gamma <2,$ then $\left( \rho _{n}\left( 4\right) ,\rho _{n}\left(
5\right) \right) \rightarrow \left( +\infty ,+\infty \right)$.\\

\noindent b) If $\gamma >2,$ then $\left( \rho _{n}\left( 4\right) ,\text{ }\rho
_{n}\left( 5\right) \right) \rightarrow \left( 0,0\right)$.\\

\noindent c) If $\gamma =2,$ both limits and any other one is possible.\\
\end{lemma}

\begin{proof} of Lemma \ref{l6.1}.

Let $F\in D\left( \Lambda \right) \cup D\left( \phi \right) $. Thus $G\in
D\left( \Lambda \right) $ and by (\ref{4.27}), for any $\varepsilon ,$ $%
0<\varepsilon$. Then for large values of  $n$, 
\begin{equation}
\left( z_{n}-x_{n}\right) /R_{1}\left( z_{n}\right) \leq 3\left\{ s\left(
k/n\right) \text{ }s\left( \ell /n\right) \right\} +3\left\{ \left(
k/n\right) ^{\varepsilon }\log \left( k/\ell \right) \right\} +1.
\label{6.12}
\end{equation}

\noindent This and (\ref{4.21}) together ensure that $\rho _{n}\left( 4\right)
\rightarrow 0$ ($\ell $ fixed) and $\rho _{n}\left( 5\right) \rightarrow 0$ $%
\left( \ell \rightarrow +\infty \right) $. Now let $F\in D\left( \psi
_{\gamma }\right) ,$ i.e., $G\in D\left( \psi _{\gamma }\right) $. By Lemma
\ref{l4.1}, 
\begin{equation}
\left( z_{n}-x_{n}\right) /R_{1}\left( z_{n}\right) \sim \left( \gamma
+1\right) \left( k/\ell \right) ^{1/\gamma }\exp \left( \int_{\ell
/n}^{k/n}b\left( t\right) \text{ }t^{-1}dt\right) .  \label{6.13}
\end{equation}

\noindent a) Let $\gamma >2.$ For any $\varepsilon ,$ $0<\varepsilon <\min \left(
1/\gamma ,1/2-1/\gamma \right) $, one has for large values of $n$, 
\begin{equation}
\rho _{n}\left( 4\right) \leq 2\left( \gamma +1\right) \ell ^{\left(
\varepsilon -1/\gamma \right) }k^{-\left( 1/2-\varepsilon -1/\gamma \right)
};\text{ }\rho _{n}\left( 5\right) \leq 2\left( \gamma +1\right) \left( \ell
/k\right) ^{1/2-\varepsilon -1/\gamma }  \label{6.14}
\end{equation}
which both imply that $\left( \rho _{n}\left( 4\right) ,\rho _{n}\left(
5\right) \right) \rightarrow \left( 0,0\right)$.\\

\noindent b) Let $\gamma <2$. For $\varepsilon ,$ $0<\varepsilon <\min \left( 1/\gamma
,-1/2+1/\gamma \right) $, one has as $n$ is large, which both imply that $%
\left( \rho _{n}\left( 4\right) ,\rho _{n}\left( 5\right) \right)
)\rightarrow \left( +\infty ,+\infty \right)$.\\

\noindent c) Let $\gamma =2$.\\

$\alpha $) If $b\left( t\right) =t,$ $\rho _{n}\left( 4\right) \rightarrow
\left( \gamma +1\right) \ell ^{-1/2},\rho _{n}\left( 5\right) \rightarrow
\left( \gamma +1\right)$.\\

$\beta $) If $b\left( t\right) =1/\log \log \left( 1/t\right) ,k\sim
n^{\upsilon },$ $0<\upsilon <1,$ $\left( \rho _{n}\left( 4\right) ,\rho
_{n}\left( 5\right) \right) \rightarrow \left( 0,0\right)$.\\

$\gamma $) If $b\left( t\right) =-1/\log \log \left( 1/t\right) ,k\sim
n^{\upsilon },$ $0<\upsilon <1,$ $\left( \rho _{n}\left( 4\right) ,\rho
_{n}\left( 5\right) \right) \rightarrow \left( +\infty ,+\infty \right)$.\\

\noindent Points ($\alpha $), ($\beta $) and ($\gamma $) prove that any limit is possible when $\gamma =2$.
\end{proof}

\bigskip

\noindent Here are the results for $T_{n}\left( 3,k,\ell ,\upsilon \right) $.

\begin{theorem} \label{t6.3}. Let $F\equiv \left( f,b\right) \in \Gamma $, and $\left( k,\ell \right) 
$ satisfies (\ref{2.4a}) or (\ref{2.4b}). Suppose that the regularity conditions holds.

\noindent 1) Let $F\in D\left( \Lambda \right) \cup D\left( \phi \right) .$

\noindent a) If $\ell $ is fixed, then

$$
R_{1}\left( z_{n}\right) ^{-1}R_{1}\left( x_{n}\right) \left( n^{\upsilon
}T_{n}\left( 3,k,\upsilon \right) -\frac{z_{n}-x_{n}}{\mu \left( k\right) }%
\right)
$$

$$
=-\frac{z_{n}-\widetilde{z}_{n}}{R_{1}\left( z_{n}\right) }%
+o_{p}\left( 1\right) \rightarrow ^{d}-\log E\left( \ell \right).
$$

\noindent b) If $\ell \rightarrow +\infty $, then

$$
v_{n}\left( 3\right) \left( n^{\upsilon }T_{n}\left( 3,k,\ell ,\upsilon
\right) -c_{n}\left( 3\right) \right) -\frac{\ell ^{1/2}\left( z_{n}-%
\widetilde{z}_{n}\right) }{R_{1}\left( z_{n}\right) }+o_{p}\left( 1\right)
$$

$$
=-N_{n}\left( 2,\ell \right) +o_{p}\left( 1\right) \rightarrow N\left(
0,1\right),
$$

\noindent where $v_{n}\left( 3\right) =\ell ^{1/2}R_{1}\left( x_{n}\right)
/R_{1}\left( z_{n}\right)$, and  $c_{n}\left( 3\right) =\left(
z_{n}-x_{n}\right) /\mu \left( k\right)$.\\

\noindent 2) Let $F\in D\left( \psi _{\gamma }\right) $ with $\gamma >2$, then

$$
R_{1}\left( x_{n}\right) \left( n^{\upsilon }T_{n}\left( 3,k,\ell ,\upsilon
\right) -c_{n}\left( 3\right) \right) =e_{1}\left( \gamma \right)
N_{n}\left( 2,\ell \right) +o_{p}\left( 1\right)
$$ 

\noindent when $\ell \rightarrow +\infty $.\\

\noindent 3) Let $F\in D\left( \psi _{\gamma }\right) $ with $0<\gamma <2$, then for $%
v_{n}^{\ast }\left( 3\right) =k^{1/2},$

$$
v_{n}^{\ast }\left( 3\right) \left( nT_{n}^{\upsilon }\left( 3,k,\ell
,\upsilon \right) -c_{n}\left( 3\right) \right) =-\left( \gamma +1\right)
N_{n}\left( 0,k,\ell \right)
$$

$$
-N_{n}\left( 2,\ell \right) +o_{p}\left(
1\right) \rightarrow ^{d}N\left( 0,\left( \gamma +1\right) ^{2}\left(
5\gamma +8\right) /\left( \gamma +2\right) \right),
$$

\noindent in both cases where $\ell $ is fixed and $\ell $ tends to infinity.
\end{theorem}

\begin{remark} \label{r6.1} We may obtain a mixture case for $\gamma =2$. For instance, put $%
b\left( t\right) =t^{a},$ $a>1$, in (\ref{2.6}) Then, $\rho _{n}\left(
2,1/2\right) \rightarrow \ell ^{-1/2}=b_{o},\rho _{n}\left( 4\right)
\rightarrow 3\ell ^{-1/2}=b_{1}$ and $\rho _{n}\left( 5\right) \rightarrow
3=b_{2}$ so that.\\

$$
R_{1}\left( z_{n}\right) ^{-1}R_{1}\left( x_{n}\right) \left( n^{\upsilon
}T_{n}\left( 3,k,\ell ,\upsilon \right) -c_{n}\left( 3\right) \right)
=3b_{o}N_{n}\left( 2,k\right) /2
$$

$$
-b_{1}N_{n}\left( 1,k,\ell \right)+b_{1}1-E_{n}\left( \ell \right) ^{1/2}+o_{p}\left( 1\right)
$$

\noindent and\\

\noindent 
$$
\ell ^{1/2}R_{1}\left( z_{n}\right) ^{-1}R_{1}\left( x_{n}\right) \left(
n^{\upsilon }T_{n}\left( 3,k,\ell ,\upsilon \right) -c_{n}\left( 3\right)
\right) =3N_{n}\left( 2,k\right) /2
$$

$$
-3N_{n}\left( 1,k,\ell ,\upsilon \right)+o_{p}\left( 1\right).
$$
\end{remark}

\bigskip

\begin{proof} Set $T_{n}^{\ast }\left( 3\right) =n^{\upsilon }T_{n}\left( 3,k,\ell
,\upsilon \right) -c_{n}\left( 3\right) $. By Theorem \ref{t4.1}, 

\begin{equation*}
T_{n}^{\ast} \left( 3\right) =\mu \left( k\right) ^{-1} \left( {\left( z_{n}-\widetilde{z}_{n}\right) +\left( x_{n}-\widetilde{x}_{n}\right)} \right. 
\end{equation*}

\begin{equation}
\left. {-\sqrt{k} \left( z_{n}-x_{n}\right) N_{n}\left( 1,k,\ell \right) +o_{p}\left( \left(
z_{n}-x_{n}\right) /\sqrt{k}\right)} \right).  \label{6.16}
\end{equation}

\bigskip

\noindent Hence, 

\begin{equation*}
R_{1}\left( x_{n}\right) R_{1}\left( z_{n}\right) ^{-1}T_{n}^{\ast }\left(
3\right) =-\frac{z_{n}-\widetilde{z}_{n}}{R_{1}\left( z_{n}\right) }%
+e_{1}\left( \gamma \right) \rho _{n}\left( 2,1/2\right) N_{n}\left(
2,k\right)
\end{equation*}

\begin{equation}
-\rho _{n}\left( 4\right) N_{n}\left( 1,k,\ell \right)
+o_{p}\left( \rho _{n}\left( 4\right) \right) .  \label{6.17}
\end{equation}

\noindent From this step, Lemmas \ref{l5.1} and \ref{l6.1} and the fact that $\left(
z_{n}-x_{n}\right) \sim \left( \gamma +1\right) R_{1}\left( x_{n}\right) $
give all the possibilities listed in Theorem \ref{t6.3}.
\end{proof}

\subsection{Case of $T_{n}\left( 6\right) $.}

\bigskip

\begin{theorem} \label{t6.4}. Let $F\equiv \left( f,b\right) \in \Gamma $ and $\left( k,\ell \right) $
satisfies (\ref{2.4a}) or (\ref{2.4b}) with $\ell \rightarrow +\infty $. If the regularity
conditions hold, then for $v_{n}\left( 6\right) =\ell ^{1/2}R_{1}\left(
z_{n}\right) /\left( z_{n}-x_{n}\right) ,$ $v_{n}^{\ast }\left( 6\right)
=\ell ^{1/2}\left( z_{n}-x_{n}\right) /R_{1}\left( z_{n}\right) ,$ $%
c_{n}\left( 6\right) =\left( z_{n}-x_{n}\right) /\mu \left( \ell \right) ,$%

\begin{equation}
v_{n}\left( 6\right) \left( T_{n}\left( 6\right) ^{-1}-c_{n}\left( 6\right)
\right) =-N_{n}\left( 1,\ell ,1\right) +o_{p}\left( 1\right), \label{6.18}
\end{equation}

\noindent  for $2<\gamma \leq +\infty$ and 
\begin{equation*}
v_{n}^{\ast }\left( 6\right) \left( T_{n}\left( 6\right) ^{-1}-c_{n}\left(
6\right) \right) =-N_{n}\left( 1,\ell ,1\right) +o_{p}\left( 1\right) ,\text{
for }0<\gamma <2.
\end{equation*}
\end{theorem}

\bigskip

\begin{proof} This is a simple case of the next proof.
\end{proof}

\subsection{Case of $T_{n}\left( 7\right) $}

\bigskip

\begin{theorem} \label{t6.5}. Let $F\equiv \left( f,b\right) \in \Gamma $ and $\left( k,\ell \right) $
satisfies (\ref{2.4b}) or (\ref{2.4b}) with $\ell \rightarrow +\infty $. Suppose that regularity conditions hold and put

$$
v_{n}\left( 7\right) =\ell ^{1/2}R_{2}\left( z_{n}\right) /\left(
z_{n}-x_{n}\right) ^{2},
$$

\noindent 
$$v_{n}^{\ast }\left( 7\right) =\ell ^{1/2}\left( z_{n}-x_{n}\right)
^{2}/R_{2}\left( z_{n}\right),
$$ 

\noindent 

$$
c_{n}\left( 7\right) =\left(z_{n}-x_{n}\right) ^{2}/\tau \left( \ell \right).
$$

\noindent Then we have\\

$$
v_{n}\left( 7\right) \left( T_{n}\left( 7\right) ^{-1}-c_{n}\left( 7\right)
\right) =-N_{n}\left( 4,\ell ,1\right) +o_{p}\left( 1\right)
$$ 

\noindent  for $2<\gamma\leq +\infty$ and\\

$$
v_{n}^{\ast }\left( 7\right) \left( T_{n}\left( 7\right) ^{-1}-c_{n}\left(
7\right) \right) =N_{n}\left( 4,\ell ,1\right) +o_{p}\left( 1\right)
$$

\noindent for $0<\gamma<2$.
\end{theorem}

\begin{proof} By (\ref{5.1}), in the case where $2<\gamma \leq +\infty$ for instance,

$$\widetilde{z}_{n}-\widetilde{x}_{n}=\left( z_{n}-x_{n}\right) -R_{1}\left(
z_{n}\right) \ell ^{1/2}e_{1}\left( \gamma \right) N_{n}\left( 2,\ell
\right) +o_{p}\left( R_{1}\left( z_{n}\right) \ell ^{-1/2}\right).
$$

\noindent Thus

$$
\left( \widetilde{z}_{n}-\widetilde{x}_{n}\right) ^{2}=\left(
z_{n}-x_{n}\right) ^{2}-2\left( z_{n}-x_{n}\right) R_{1}\left( z_{n}\right)
\ell ^{-1/2}e_{1}\left( \gamma \right) N_{n}\left( 2,\ell \right)
$$

$$
+o_{p}\left( \left( z_{n}-x_{n}\right) R_{1}\left( z_{n}\right) \ell
^{-1/2}\right).
$$

\noindent The term $o_{p}\left( \left( z_{n}-x_{n}\right) R_{1}\left( z_{n}\right)
\ell ^{-1/2}\right) $ is justified by the fact that 
$$
R_{1}\left( z_{n}\right) /\left( z_{n}-x_{n}\right) \rightarrow 0
$$ 

\noindent (See Lemma \ref{l4.2}). Hence, by using Theorem \ref{t5.1}, we arrive at 

\begin{equation*}
\ell ^{1/2}\left\{ T_{n}\left( 7\right) ^{-1}-\frac{\left(
z_{n}-x_{n}\right) ^{2}}{\tau \left( \ell \right) }\right\} =\ell
^{1/2}\left\{ \frac{\left( \widetilde{z}_{n}-\widetilde{x}_{n}\right) ^{2}}{%
A_{n}\left( 1,k,\ell \right) }-\frac{\left( z_{n}-x_{n}\right) }{\tau \left(
\ell \right) }\right\}
\end{equation*}

\begin{equation*}
=\kappa \left( \gamma \right) ^{-1}\left\{ -2\frac{z_{n}-x_{n}}{R\left(
z\right) }e_{1}\left( \gamma \right) N_{n}\left( 2,\ell \right) -\left(
\left( z_{n}-x_{n}\right) /R_{1}\left( z_{n}\right) \right) ^{2}N_{n}\left(
4,\ell ,1\right) \right\} 
\end{equation*}

\begin{equation*}
+o_{p}\left( \frac{z_{n}-z_{n}}{R_{1}\left( z_{n}\right) }\right)
+o_{p}\left( \frac{\left( z_{n}-x_{n}\right) ^{2}}{R_{1}\left( z_{n}\right)
^{2}}\right) ,
\end{equation*}

\noindent which in turn implies

$\left( z_{n}-x_{n}\right) ^{-2}R_{2}\left( z_{n}\right) \ell ^{1/2}\left(
T_{n}\left( 7\right) ^{-1}-c_{n}\left( 7\right) \right) =-N_{n}\left( 4,\ell
,1\right) +o_{p}\left( 1\right) .$

\noindent By Using again (\ref{5.1}) for $0<\gamma <2$ (that is $\rho _{n}\left(
i,1/2\right) \rightarrow +\infty $, $i=2,3$) gives the result for the last
case. The proof is now complete.
\end{proof}

\noindent Finally, we give the multivariable version of all that preceedes.

\section{Multivariable asymptotic normality of the characterizing vectors.} \label{sec7}

\noindent We neglect $T_{n}\left( 4\right) $ since its asymptotic law is extremal.
Remark that by Corollary \ref{c3.3}, 
\begin{equation}
v_{n}\left( 5\right) \left( T_{n}\left( 2,\ell ,1\right) -c_{n}\left(
5\right) \right) =N_{n}\left( 1,\ell ,1\right) +o_{p}\left( 1\right) ,
\label{7.1}
\end{equation}
where $v_{n}\left( 5\right) =\ell ^{1/2}/R_{1}\left( z_{n}\right)$, $c_{n}\left( 5\right) =\mu \left( \ell \right) $, when $F\equiv \left(f,b\right) \in \Gamma $, $\left( k,\ell \right) $ satisfies (\ref{2.4a}) and
the regularity conditions (with respect to $\ell $) hold and $\left( \ell
,\ell /k\right) \rightarrow \left( +\infty ,0\right) $. Now set $T^{\ast}_{n}(\infty)$ as the following vector, in
the case $\gamma=+\infty$.

$\bigskip $%
\begin{equation*}
\left( 
\begin{array}{c}
v_{n}\left( 1\right) \left( T_{n}\left(1,k,\ell \right) - c_{n}\left( 1\right) \right)\\ 
v_{n}\left( 0\right) \left( A_{n}\left( 1,k,\ell \right)  - c_{n}\left( 0\right) \right) \\ 
v_{n}\left( 2\right) \left( T_{n}\left( 2,k,\ell \right) - c_{n}\left( 2\right) \right) \\ 
v_{n}\left( 3\right) \left( n^{\upsilon }T_{n}\left( 3,k,\ell ,\upsilon\right) -c_{n}\left( 3\right) \right) \\ 
v_{n}\left( 5\right) \left( T_{n}\left( 5\right) - c_{n}\left( 5\right) \right) \\ 
v_{n}\left( 6\right) \left( T_{n}\left( 6\right) ^{-1} -c_{n}\left( 6\right) \right) \\ 
v_{n}\left( 7\right) \left( T_{n}\left( 7\right) ^{-1} -c_{n}\left( 7\right) \right) \\ 
v_{n}\left( 8\right) \left( v_{n}^{-\upsilon }T_{n}\left( 8,\upsilon \right)- c_{n}\left( 8\right) \right)
\end{array}
\right)
\end{equation*}

\bigskip

\noindent Alternatively, when $\gamma>2$, we define $T_{n}^{\ast }\left(\gamma >2\right)$ from $T^{\ast}_{n}(\infty)$ by replacing $v_{n}\left(8,\upsilon \right) ^{-1}-c_{n}\left( 8\right) $ by $v_{n}\left( 9\right)
\left( T_{n}\left( 9\right) -c_{n}\left( 9\right) \right)$.\\

\noindent When $\gamma<2$, $T_{n}^{\ast }\left( \gamma <2\right) $ is defined from $T_{n}^{\ast }\left(
\gamma >2\right) $ by replacing $v_{n}\left( 3\right) $, $v_{n}\left(6\right) $ and $v_{n}\left( 7\right) $ by $v_{n}^{\ast }\left( 3\right) $, $v_{n}^{\ast }\left( 6\right) $ and $v_{n}^{\ast }\left( 7\right) $
respectively.\\

\noindent  We obtain

\begin{theorem} \label{t7.1}. Let $\equiv \left( f,b\right) \in \Gamma $ and $\left( k,\ell \right) $
satisfies (\ref{2.4a}) or (\ref{2.4b}). We assume that the regularity conditions hold.\\

\noindent a) If $F\in D\left( \phi \right) \cup D\left( \Lambda \right) $, then $%
T_{n}^{\ast }\left( \infty \right) +\sum_{n}\left( \infty \right)
+o_{p}\left( 1\right)$, 

\noindent where $^{t}\sum_{n} \left( \infty \right)=\left( \sum_{n}^{\ast }\left( \infty \right) \right) $, $\sum_{n}^{\ast\ast }\left( \infty \right)$ is an $\mathbb{R}^{8}$-Gaussian vector such that $%
\sum_{n}^{\ast }\left( \infty \right) $ (an $\mathbb{R}^{3}$-$rv$) and $\sum_{n}^{\ast }\left( \infty \right) $ (an $\mathbb{R}^{5}$-$rv$) are asymptotically independent with respective limiting covariance matrices.\\

\begin{center}
$\sum^{\ast}\left( \infty \right) =\left( 
\begin{array}{ccc}
1/2 &  &  \\ 
-1/2 & 5 &  \\ 
0 & 2 & 1
\end{array}
\right) $ 
\end{center}

\noindent and 

\begin{center}
\noindent $\sum^{\ast \ast }\left( \infty \right) =\left( 
\begin{array}{ccccc}
1 &  &  &  &  \\ 
0 & 1 &  &  &  \\ 
0 & -1 & 1 &  &  \\ 
0 & -2 & 2 & 5 &  \\ 
1 & 0 & 0 & 0 & 1
\end{array}
\right) ,$
\end{center}

\bigskip

\noindent with $^{t}M$ denoting the transpose of the matrix $M$.\\

\noindent b) If $F\in D\left( \psi _{\gamma }\right)$ , $\gamma >2$, 
then $T_{n}^{\ast}\left( \gamma >2\right) =\sum_{n}\left( \gamma >2+o_{p}\left( 1\right)
\right)$,\\ 

\noindent where $^{t}\sum_{n}\left( \gamma >2\right) =\left( \sum_{n}^{\ast
}\left( \gamma >2\right) \right),\sum_{n}^{\ast \ast }\left( \gamma >2\right) $ is an $\mathbb{R}^{8}$-Gaussian
vector such that $\sum_{n}^{\ast }\left( \gamma >2\right) $ (an $\mathbb{R}^{3}$%
-$rv$) and $\sum_{n}^{\ast \ast }\left( \gamma >2\right) $ (an $\mathbb{R}^{5}$-$rv$)
are asymptotically independent with respective limiting covariance matrices : $\sum^{\ast}\left( \gamma >2\right)=$\\

\begin{center}
$ \left( 
\begin{array}{cccc}
\frac{\gamma ^{3}+\gamma ^{2}+2\gamma }{4\left( \gamma +1\right) \left(
\gamma +3\right) \left( \gamma +4\right) } &  &  &  \\ 
-1/2\left( \frac{\gamma +2}{\gamma +1}\right) ^{1/2} & \frac{\gamma \left(
\gamma -5\right) }{\left( \gamma +3\right) \left( \gamma +4\right) } & \frac{%
5\gamma +11\gamma +4^{4}\gamma +7^{3}\gamma +12^{2}}{\gamma ^{2}\left(
\gamma +3\right) \left( \gamma +4\right) } &  \\ 
1/2\left( \frac{\gamma +2}{\gamma +1}\right) ^{1/2} & \frac{2\gamma }{\left(
\gamma +1\right) \left( \gamma +2\right) } & \frac{2\gamma ^{3}+4\gamma
^{2}+18\gamma +18}{\gamma ^{2}\left( \gamma +3\right) } & \frac{\gamma
^{3}+\gamma ^{2}+2}{\gamma ^{2}\left( \gamma +2\right) }
\end{array}
\right)$
\end{center}

\noindent and \\

\begin{center}
$\sum^{\ast \ast }\left( \gamma >2\right) =\left( 
\begin{array}{ccccc}
\left( \gamma +1\right) /\gamma &  &  &  &  \\ 
-\left( \gamma +1\right) /\gamma ^{2} & \frac{\gamma ^{3}+\gamma ^{2}+2}{%
\gamma ^{2}+\left( \gamma +2\right) } &  &  &  \\ 
\left( \gamma +1\right) /\gamma ^{2} & \frac{\gamma ^{3}+\gamma ^{2}+2}{%
\gamma ^{2}\left( \gamma +2\right) } & \frac{\gamma ^{3}+\gamma ^{2}+2}{%
\gamma ^{2}\left( \gamma +2\right) } &  &  \\ 
2\left( \gamma +1\right) \gamma ^{2} & \frac{2\gamma ^{3}+4\gamma
^{2}+18\gamma +18}{\gamma ^{2}\left( \gamma +3\right) } & \frac{2\gamma
^{3}+4\gamma ^{2}+18\gamma +18}{\gamma ^{2}\left( \gamma +3\right) } & \frac{%
5\gamma ^{4}+11\gamma ^{3}+4\gamma ^{2}+7\gamma +12}{\gamma ^{2}\left(
\gamma +3\right) \left( \gamma +4\right) } &  \\ 
-\left( \gamma +1\right) /\gamma ^{2} & \gamma ^{-2} & -\gamma ^{-2} & 
-2\gamma ^{-2} & \gamma ^{-2}
\end{array}
\right)$
\end{center}

\bigskip 

\noindent c) If $F\in D\left( \psi _{\gamma }\right) ,$ then $T_{n}^{\ast }\left(
\gamma <2\right) =\sum_{n}\left( \gamma <2\right) +o_{p}\left( 1\right)$, where 
$$
^{t}\sum_{n}\left( \gamma <2\right) \left( \sum_{n}^{2}\left( \gamma
<2\right) ,\sum_{n}^{\ast \ast }\left( \gamma <2\right) \right) 
$$ 

\noindent is an $\mathbb{R}^{8}$-Gaussian vector such that $\sum_{n}^{\ast }\left( \gamma <2\right)$
(an $\mathbb{R}^{4}$-$rv$) and $\sum_{n}^{\ast \ast }\left( \gamma <2\right) $ (an $\mathbb{R}
^{4}$-$rv$) are asymptotically independent with respective covariance
matrices  $\sum^{\ast }\left( \gamma <2\right)=$

\begin{center}
\small
$\left( 
\begin{array}{ccccc}
\frac{\gamma ^{3}+\gamma ^{2}+2}{4\left( \gamma +1\right) \left( \gamma
+3\right) \left( \gamma +4\right) } &  &  &  &  \\ 
-1/2\left( \frac{\gamma +2}{\gamma +1}\right) ^{1/2} & \frac{\gamma \left(
\gamma +5\right) }{\left( \gamma +3\right) \left( \gamma +4\right) } & \frac{%
5\gamma ^{4}+11\gamma ^{3}+4\gamma ^{2}+7\gamma +12}{\gamma ^{2}\left(
\gamma +3\right) \left( \gamma +4\right) } &  &  \\ 

-1/2\left( \frac{\gamma +2}{\gamma +1}\right) & \frac{2\gamma }{\left(
\gamma +3\right) \left( \gamma +4\right) } & \frac{2\gamma ^{3}+4\gamma
^{2}+18\gamma +18}{\gamma ^{2}\left( \gamma +3\right) } & \frac{\gamma
^{3}+\gamma ^{2}+2}{\gamma ^{2}\left( \gamma +2\right) } &  \\ 

-1/2\left( \gamma +2\right) \left( \gamma +1\right) ^{1/2} & \frac{2\gamma
^{2}+12\gamma +12}{\left( \gamma +2\right) \left( \gamma +3\right) } -\left( \gamma +1\right)& 
 \frac{2\gamma ^{2}+4\gamma ^{2}-12}{\gamma \left(
\gamma +3\right) }-\left( \gamma +1\right) &  \frac{\gamma ^{2}-2\gamma -4}{%
\gamma \left( \gamma +2\right) } & \left( \gamma +1\right) ^{2}\frac{5\gamma
+8}{\gamma +2}
\end{array}
\right)$
\end{center}

\bigskip

\noindent and $\sum^{\ast \ast }\left( \gamma <2\right)$\\

\bigskip

\begin{center}
$=\left( 
\begin{array}{cccc}
\frac{\gamma ^{3}+\gamma ^{2}+2}{\gamma ^{2}\left( \gamma +2\right) } &  & 
&  \\ 
-\frac{\gamma ^{3}+\gamma ^{2}+2}{\gamma ^{2}\left( \gamma +2\right) } & 
\frac{\gamma ^{3}+\gamma ^{2}+2}{\gamma ^{2}\left( \gamma +2\right) } &  & 
\\ 
-\frac{2\gamma ^{3}+4\gamma ^{2}+18\gamma +18}{\gamma ^{2}\left( \gamma
+3\right) } & \frac{2\gamma ^{3}+4\gamma ^{2}+18\gamma +18}{\gamma
^{2}\left( \gamma +1\right) } & \frac{5\gamma ^{4}+11\gamma ^{3}+4\gamma
+7\gamma +12}{\gamma ^{2}\left( \gamma +3\right) \left( \gamma +4\right) } & 
\\ 
\gamma ^{-2} & -\gamma ^{-2} & -2\gamma ^{-2} & \gamma ^{-2}
\end{array}
\right)$.
\end{center}

\end{theorem}

\bigskip

\begin{proof}

\noindent Putting together Corollary \ref{c3.2}, Theorem \ref{t4.1}, Corollary \ref{c5.1}, Theorems \ref{t6.3},
\ref{t6.4} and Formula \ref{7.1}, one gets that $=\sum_{n}^{\ast}\left( \infty \right)$ is the vector

$$
0.5\left( \frac{\gamma +2}{\gamma +1}%
\right) ^{1/2} \left( {2N_{n}\left( 0,k,\ell \right) -N_{n}\left( 3,k,\ell
\right) +N_{n}\left( 2,k\right),} \right.
$$

$$
\left. {N_{n}\left( 3,k, \ell \right)+e_{3}\left( \gamma \right) N_{n}\left( 2,k\right), N_{n}\left( 0,k,\ell \right) +e_{1}\left( \gamma \right) N_{n}\left( 2,k\right)} \right).
$$

\noindent while $\sum_{n}^ {\ast \ast }\left( \infty \right)$ is the vector
$$
\left( {e_{1} \left( \gamma
\right) N_{n}\left( 0,\ell ,1\right) +e_{1}\left( \gamma \right) N_{n}\left(
2,\ell \right) ,-N_{n}\left( 0,\ell ,1\right) -e_{1}\left( \gamma \right)
N_{n}\left( 2,\ell \right),} \right.
$$

$$
\left. {-N_{n} \left( 3,\ell ,1\right)
-e_{3}\left( \gamma \right) N_{n}\left( 2,\ell \right) ,-N_{n}\left( 2,\ell\right)} \right.
$$

\noindent with $\gamma =+\infty$; $\sum_{n}^{\ast }\left( \gamma >2\right)
=\sum_{n}^{\ast }\left( \infty \right) $ with $2<\gamma <+\infty $ and $%
\sum_{n}^{\ast \ast }\left( \gamma >2\right) $ is obtained from $%
\sum_{n}^{\ast \ast }\left( \infty \right) $ by replacing the last line by $%
N_{n}\left( 2,\ell \right) /\gamma$.\\

\noindent  $\sum_{n}^{\ast }\left( \gamma <2\right) $ is obtained from $\sum_{n}^{\ast
}\left( \infty \right) $ by adding $-\left( \gamma +1\right) N_{n}\left(
0,k,\ell \right) -N_{n}\left( 2,k\right) $ as a fourth line.\\

\noindent Finally, one forms $\sum_{n}^{\ast }\left( \gamma <2\right) $ by dropping the first line
of $\sum_{n}^{\ast \ast }\left( \gamma \geq 2\right) $.\\

\noindent Now simple computations show that if $L_{n}\left( 1\right) $ (resp. $L_{n}\left(
2\right) $) is any coordinate of $\sum_{n}^{\ast \ast }\left( \infty \right) 
$ (resp. $\sum_{n}^{\ast \ast }\left( \infty \right) $), $\sum_{n}^{\ast
}\left( \gamma >2\right) $ (resp. $\sum_{n}^{\ast \ast }\left( \gamma
>2\right) $ or $\sum^{\ast }\left( \gamma <2\right) $ (resp. $\sum_{n}^{\ast
\ast }\left( \gamma <2\right) )$, one has 

\begin{equation*}
\mathbb{E}L_{n}\left( 1\right) L_{n}\left( 2\right) \sim \rho _{n}\left(
1,1/2\right) \text{ or }\left( 2\kappa \left( \gamma \right) \right)
^{-1}\rho _{n}\left( 1,1/2\right) ^{2}\text{ or }\left( \ell /k\right)
^{1/2}.
\end{equation*}

\noindent (see computations that led to (\ref{6.10})). Hence, by (\ref{4.20}),

\begin{equation}
\mathbb{E}L_{n}\left( 1\right) L_{n}\left( 2\right) \rightarrow 0. \label{7.2}
\end{equation}

\bigskip

\noindent This together with (\ref{3.34}), (\ref{3.43}), (\ref{4.22}), (\ref{4.27}), (%
\ref{6.10}), Lemmas \ref{l4.1} and \ref{l4.3} yield the covariance matrices by routine
computations.\\

\noindent For $\gamma=2$, $T_{n}^{\ast }\left( \gamma =2\right) $, several possibilities can
happen depending on how $b\left( t\right) $ converges to zero as $t$ tends
to zero. Anyway Formulas (\ref{6.17}) and (\ref{6.18}) include all the
possibilities of limiting laws.
\end{proof}

\end{document}